\documentclass[10pt,reqno]{amsart}
\usepackage[foot]{amsaddr}
\usepackage{etoolbox}
\usepackage[a4paper,margin=3cm]{geometry}
\usepackage{microtype}
\usepackage{graphicx}
\usepackage[justification=centering, margin=20pt]{caption}
\usepackage{subcaption}
\usepackage{booktabs}
\usepackage{dsfont}

\usepackage[colorlinks=true,citecolor=blue,linkcolor=blue]{hyperref}

\makeatletter
\renewcommand{\@secnumfont}{\bfseries}
\makeatother
\makeatletter
\def\subsection{\@startsection{subsection}{3}%
  \z@{.5\linespacing\@plus.7\linespacing}{.2\linespacing}%
  {\normalfont\bfseries}}
\makeatother

\renewcommand{\paragraph}[1]{\ \\ \textbf{#1}}

\usepackage[utf8]{inputenc}

\usepackage{amsthm}
\usepackage{amssymb}
\usepackage{amsfonts}

\usepackage[round]{natbib}

\newtheorem{theorem}{Theorem}
\newtheorem{lemma}{Lemma}

\newtheorem{assumption}{Assumption}
\newtheorem{corollary}{Corollary}

\usepackage{amsmath}
\usepackage{dsfont}

\usepackage{algorithm}
\usepackage{algorithmic}

\newcommand{\esp}[1]{\mathbb{E}\left[ #1 \right]}
\newcommand{\grando}[1]{O\left(#1 \right)}
\newcommand{\adfs}{ADFS}
\newcommand{\esdacd}{ESDACD}
\newcommand{\msda}{MSDA}

\newcommand{\R}{\mathbb{R}}

\title{An Accelerated Decentralized Stochastic Proximal Algorithm for Finite Sums}

\date{}

\author{
Hadrien Hendrikx $^\dagger$ 
\qquad Francis Bach \qquad Laurent Massouli\'e $^\dagger$
}

\address{INRIA - Département d’informatique de l’ENS \\
Ecole normale supérieure, CNRS, INRIA \\ 
PSL Research University, 75005 Paris, France}

\address{$\dagger$ Microsoft-INRIA Joint Centre}

\email{hadrien.hendrikx@inria.fr}
\email{francis.bach@inria.fr}
\email{laurent.massoulie@inria.fr}

\begin{document}

\maketitle

\begin{abstract}
Modern large-scale finite-sum optimization relies on two key aspects: distribution and stochastic updates. For smooth and strongly convex problems, existing decentralized algorithms are slower than modern accelerated variance-reduced stochastic algorithms when run on a single machine, and are therefore not efficient. Centralized algorithms are fast, but their scaling is limited by global aggregation steps that result in communication bottlenecks. In this work, we propose an efficient \textbf{A}ccelerated \textbf{D}ecentralized stochastic algorithm for \textbf{F}inite \textbf{S}ums named ADFS, which uses local stochastic proximal updates and randomized pairwise communications between nodes. On $n$ machines, ADFS learns from $nm$ samples in the same time it takes optimal algorithms to learn from $m$ samples on one machine. This scaling holds until a critical network size is reached, which depends on communication delays, on the number of samples $m$, and on the network topology. We provide a theoretical analysis based on a novel augmented graph approach combined with a precise evaluation of synchronization times and an extension of the accelerated proximal coordinate gradient algorithm to arbitrary sampling. We illustrate the improvement of ADFS over state-of-the-art decentralized approaches with experiments.
\end{abstract}

\section{Introduction}

The success of machine learning models is mainly due to their capacity to train on huge amounts of data. Distributed systems can be used to process more data than one computer can store or to increase the pace at which models are trained by splitting the work among many computing nodes. In this work, we focus on problems of the form:\vspace{-3pt}
\begin{equation}
\label{eq:distributed_problem}
    \min_{\theta \in \mathbb{R}^d}\ \  \sum_{i=1}^n f_i(\theta),\ \ \ \ \mbox{ where } \ \ \ \  f_i(\theta) = \sum_{j=1}^m f_{i,j}(\theta) +  \frac{\sigma_i}{2}\| \theta \|^2 . \! \! \! \! \! \! \! \! \!
\end{equation}
This is the typical $\ell_2$-regularized empirical risk minimization problem with $n$ computing nodes that have $m$ local training examples each. The function $f_{i,j}$ represents the loss function for the $j$-th training example of node $i$ and is assumed to be convex and $L_{i,j}$-smooth~\citep{nesterov2013introductory, bubeck2015convex}.
\begin{table*}[t]
\centering 
\begin{small}
\begin{sc}
\begin{tabular}{lcccc}
\toprule
Algorithm & Synchrony & Stochastic & Time \\
\midrule
Point-SAGA~\citep{defazio2016simple}  & N/A & \checkmark & $nm + \sqrt{n m\kappa_s}$\\
MSDA~\citep{scaman2017optimal} & Global & $\times$ & $\sqrt{\kappa_b}\left(m + \frac{\tau}{\sqrt{\gamma}}\right)$\\
ESDACD~\citep{hendrikx2018accelerated}  & Local & $\times$ & $\left(m + \tau \right)\sqrt{\frac{\kappa_b}{\gamma}}$\\
DSBA~\citep{shen2018towards}  & Global & \checkmark & $\left(m + \kappa_s + \gamma^{-1}\right) \left(1 + \tau\right)$\\
\adfs~(this paper) & Local & \checkmark & $m + \sqrt{m\kappa_s} + (1 + \tau)\sqrt{ \frac{\kappa_s}{\gamma}}$\\
\bottomrule
\end{tabular}
\end{sc}
\end{small}
\caption{Comparison of various state-of-the-art decentralized algorithms to reach accuracy $\varepsilon$ in regular graphs. Constant factors are omitted, as well as the $\log\left(\varepsilon^{-1}\right)$ factor in the \emph{\textsc{Time}} column. Reported runtime for Point-SAGA corresponds to running it on a single machine with $n m$ samples. To allow for direct comparison, we assume that computing a dual gradient of a function $f_i$ as required by MSDA and ESDACD takes time $m$, although it is generally more expensive than to compute $m$ separate proximal operators of single $f_{i,j}$ functions.} 
\label{fig:table_speeds}
\vskip -0.1in
\end{table*}
These problems are usually solved by first-order methods, and the basic distributed algorithms compute gradients in parallel over several machines~\citep{nedic2009distributed}. Another way to speed up training is to use \emph{stochastic} algorithms \citep{bottou2010large, defazio2014saga, johnson2013accelerating}, that take advantage of the finite sum structure of the problem to use cheaper iterations while preserving fast convergence. This paper aims at bridging the gap between stochastic and decentralized algorithms when local functions are smooth and strongly convex. In the rest of this paper, following Scaman et al. ~\citep{scaman2017optimal}, we assume that nodes are linked by a communication network and can only exchange messages with their neighbours. We further assume that each communication takes time~$\tau$ and that processing one sample, \emph{i.e.}, computing the proximal operator for a \emph{single} function $f_{i,j}$, takes time $1$. The proximal operator of a function $f_{i,j}$ is defined by ${\rm prox}_{\eta f_{i,j}}(x) = \arg \min_v \frac{1}{2\eta}\|v - x\|^2 + f_{i,j}(v)$. The condition number of the Laplacian matrix of the graph representing the communication network is denoted $\gamma$. This natural constant appears in the running time of many decentralized algorithms and is for instance of order $O(1)$ for the complete graph and $O(n^{-1})$ for the 2D grid. More generally, $\gamma^{-1/2}$ is typically of the same order as the diameter of the graph. Following notations from Xiao et al. \citep{xiao2017dscovr}, we define the batch and stochastic condition numbers $\kappa_b$ and $\kappa_s$ (which are classical quantities in the analysis of finite sum optimization) such that for all $i$, $\kappa_b \geq M_i / \sigma_i$ where $M_i$ is the smoothness constant of the function~$f_i$ and $\kappa_s \geq \kappa_i$, with $\kappa_i = 1 + \sum_{j=1}^m L_{i,j} / \sigma_i$ the stochastic condition number of node $i$. Although~$\kappa_s$ is always bigger than $\kappa_b$, it is generally of the same order of magnitude, leading to the practical superiority of stochastic algorithms. The next paragraphs discuss the relevant state of the art for both distributed and stochastic methods, and Table~\ref{fig:table_speeds} sums up the speeds of the main decentralized algorithms available to solve Problem~\eqref{eq:distributed_problem}. Although it is not a distributed algorithm, Point-SAGA~\citep{defazio2016simple}, an optimal single-machine algorithm, is also presented for comparison.

\paragraph{\emph{Centralized} gradient methods.}
A simple way to split work between nodes is to distribute gradient computations and to aggregate them on a parameter server. Provided the network is fast enough, this allows the system to learn from the datasets of $n$ workers in the same time one worker would need to learn from its own dataset. Yet, these approaches are very sensitive to stochastic delays, slow nodes, and communication bottlenecks. Asynchronous methods may be used~\citep{recht2011hogwild, leblond2016asaga, xiao2017dscovr} to address the first two issues, but computing gradients on older (or even inconsistent) versions of the parameter harms convergence~\citep{chen2016revisiting}. Therefore, this paper focuses on decentralized algorithms, which are generally less sensitive to communication bottlenecks~\citep{lian2017can}.\vspace{-3pt}

\paragraph{\emph{Decentralized} gradient methods.} In their synchronous versions, decentralized algorithms alternate rounds of computations (in which all nodes compute gradients with respect to their local data) and communications, in which nodes exchange information with their direct neighbors~\citep{duchi2012dual,shi2015extra, nedic2017achieving, tang2018d, he2018cola}. Communication steps often consist in averaging gradients or parameters with neighbours, and can thus be abstracted as multiplication by a so-called gossip matrix. MSDA~\citep{scaman2017optimal} is a batch decentralized synchronous algorithm, and it is optimal with respect to the constants $\gamma$ and $\kappa_b$, among batch algorithms that can only perform these two operations. Instead of performing global synchronous updates, some approaches inspired from gossip algorithms~\citep{boyd2006randomized} use randomized pairwise communications~\citep{nedic2009distributed, johansson2009randomized, colin2016gossip}. This for example allows fast nodes to perform more updates in order to benefit from their increased computing power. These randomized algorithms do not suffer from the usual worst-case analyses of bounded-delay asynchronous algorithms, and can thus have fast rates because the step-size does not need to be reduced in the presence of delays. For example, ESDACD~\citep{hendrikx2018accelerated} achieves the same optimal speed as MSDA when batch computations are faster than communications ($\tau > m$). However, both use gradients of the Fenchel conjugates of the full local functions, which are generally much harder to get than regular gradients.\vspace{-3pt}

\paragraph{Stochastic algorithms for finite sums.}
All distributed methods presented earlier are \emph{batch} methods that rely on computing \emph{full gradient} steps of each function $f_i$. Stochastic methods perform updates based on randomly chosen functions $f_{i,j}$. In the smooth and strongly convex setting, they can be coupled with \emph{variance reduction}~\citep{schmidt2017minimizing,shalev2013stochastic, johnson2013accelerating, defazio2014saga} and \emph{acceleration}, to achieve the $m + \sqrt{m\kappa_s}$ optimal finite-sum rate, which greatly improves over the $m\sqrt{\kappa_b}$ batch optimum when the dataset is large. Examples of such methods include Accelerated-SDCA~\citep{shalev2014accelerated}, APCG~\citep{lin2015accelerated}, Point-SAGA~\citep{defazio2016simple} or Katyusha~\citep{allen2017katyusha}.\vspace{-3pt}

\paragraph{Decentralized stochastic methods.} In the smooth and strongly convex setting, DSA~\citep{mokhtari2016dsa} and later DSBA~\citep{shen2018towards} are two linearly converging stochastic decentralized algorithms. DSBA uses the proximal operator of individual functions $f_{i,j}$ to significantly improve over DSA in terms of rates. Yet, DSBA does not enjoy the $\sqrt{m\kappa_s}$ accelerated rates and needs an excellent network with very fast communications. Indeed, nodes need to communicate each time they process a single sample, resulting in many communication steps. CHOCO-SGD~\citep{koloskova2019decentralized} is a simple decentralized stochastic algorithm with support for compressed communications. Yet, it is not linearly convergent and it requires to communicate between each gradient step as well. Therefore, to the best of our knowledge, there is no decentralized stochastic algorithm with accelerated linear convergence rate or low communication complexity without sparsity assumptions (\emph{i.e.}, sparse features in linear supervised learning).\vspace{-3pt}

\paragraph{ADFS.}The main contribution of this paper is a locally synchronous \textbf{A}ccelerated \textbf{D}ecentralized stochastic algorithm for \textbf{F}inite \textbf{S}ums, named ADFS. It is very similar to APCG for empirical risk minimization in the limit case $n=1$ (single machine), for which it gets the same $m + \sqrt{m\kappa_s}$ rate. Besides, this rate stays unchanged when the number of machines grows, meaning that ADFS can process $n$ times more data in the same amount of time on a network of size $n$. This scaling lasts as long as $(1 + \tau)\sqrt{\kappa_s} \gamma^{-\frac{1}{2}} < m + \sqrt{m \kappa_s}$. This means that ADFS is at least as fast as MSDA unless both the network is extremely fast (communications are faster than evaluating a single proximal operator) and the diameter of the graph is very large compared to the size of the local finite sums. Therefore, \adfs~outperforms \msda~and DSBA in most standard machine learning settings, combining optimal network scaling with the efficient distribution of optimal sequential finite-sum algorithms. Note however that, similarly to DSBA and Point-SAGA, ADFS requires evaluating ${\rm prox}_{f_{i,j}}$, which requires solving a local optimization problem. Yet, in the case of linear models such as logistic regression, it is only a constant factor slower than computing $\nabla f_{i,j}$, and it is especially much faster than computing the gradient of the conjugate of the full dual functions $\nabla f_i^*$ required by ESDACD and MSDA.
\vspace{3pt}

ADFS is based on three novel technical contributions: (i) a novel augmented graph approach which yields the dual formulation of Section~\ref{sec:model}, (ii) an extension of the APCG algorithm to arbitrary sampling that is applied to the dual problem in order to get the generic algorithm of Section~\ref{sec:alg}, and (iii) the analysis of local synchrony, which is performed in Section~\ref{sec:synch_time}. Finally, Section~\ref{sec:perfs} presents a relevant choice of parameters leading to the rates shown in  Table~\ref{fig:table_speeds}, and an experimental comparison is done in Section~\ref{sec:experiments}. A Python implementation of ADFS is also provided in supplementary material.

\section{Model and Derivations}
\label{sec:model}
\begin{figure}[!htb]
    \centering
        \centering
        \vspace{-.4cm}
          \includegraphics[width=\linewidth]{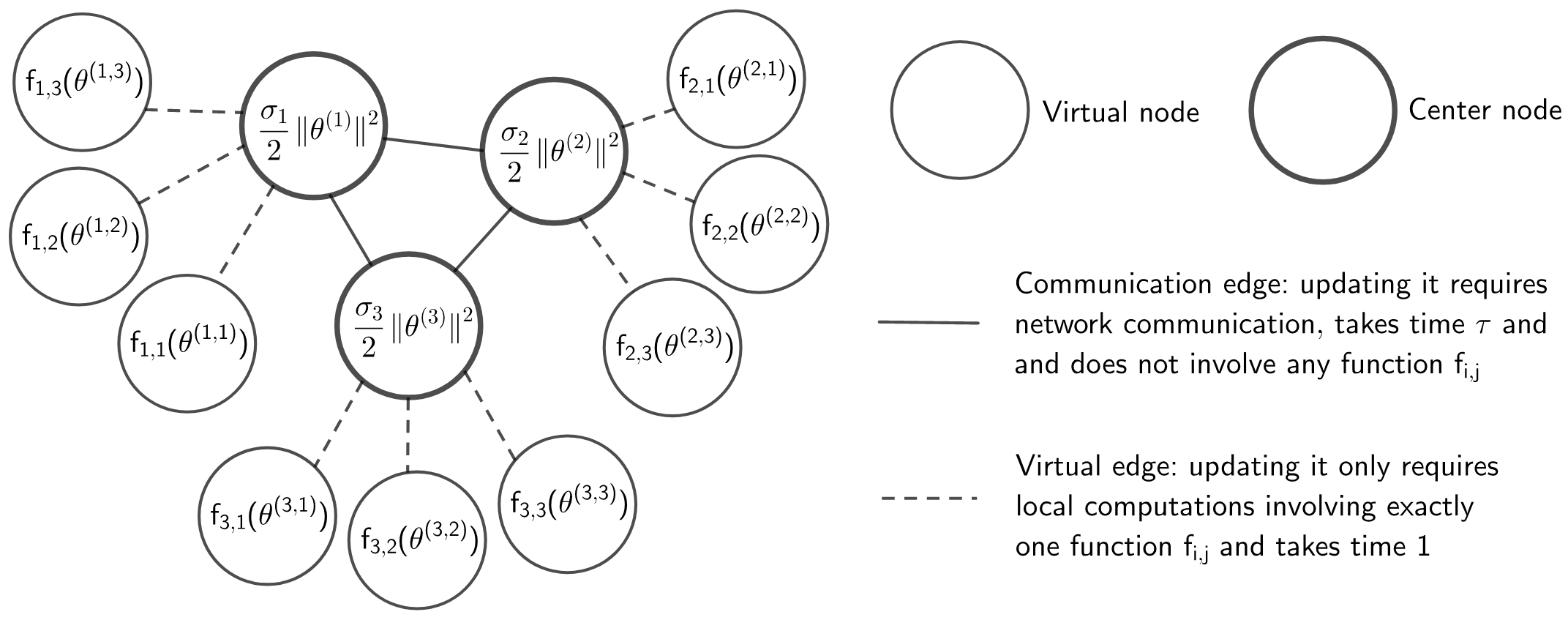}
  \vspace*{-.4cm}
\caption{Illustration of the augmented graph for $n=3$ and $m=3$.}
\label{fig:extended_graph}
\end{figure}
We now specify our approach to solve the problem in Equation~\eqref{eq:distributed_problem}. The first (classical) step consists in considering that all nodes have a local parameter, but that all local parameters should be equal because the goal is to have the global minimizer of the sum. Therefore, the problem writes: \vspace{-4pt}
\begin{equation}
    \label{eq:constr_distributed_prob}
    \min_{\theta \in \mathbb{R}^{n \times d}}\ \  \sum_{i=1}^n f_i(\theta^{(i)}) \ \  \text{ such that } \  \theta^{(i)} = \theta^{(j)} \text{ if } j \in \mathcal{N}(i),
\end{equation}
where $\mathcal{N}(i)$ represents the neighbors of node $i$ in the communication graph.
Then, \esdacd~and \msda~are obtained by applying accelerated (coordinate) gradient descent to an appropriate dual formulation of Problem~\eqref{eq:constr_distributed_prob}. In the dual formulation, constraints become variables and so updating a dual coordinate consists in performing an update along an edge of the network. In this work, we consider a new virtual graph in order to get a stochastic algorithm for finite sums. The transformation is sketched in Figure~\ref{fig:extended_graph}, and consists in replacing each node of the initial network by a star network. The centers of the stars are connected by the actual communication network, and the center of the star network replacing node~$i$ has the local function $f^{\rm comm}_i : x \mapsto \frac{\sigma_i}{2} \| x \|^2$. The center of node $i$ is then connected with $m$ nodes whose local functions are the functions $f_{i,j}$ for $j \in \{1, ..., m\}$. If we denote $E$ the number of edges of the initial graph, then the augmented graph has $n(1 + m)$ nodes and $E + nm$ edges.

Then, we consider one parameter vector $\theta^{(i,j)}$ for each function $f_{i,j}$ and one vector $\theta^{(i)}$ for each function~$f^{\rm comm}_i$. Therefore, there is one parameter vector for each node in the augmented graph. We impose the standard constraint that the parameter of each node must be equal to the parameters of its neighbors, but neighbors are now taken in the augmented graph. This yields the following minimization problem:\vspace{-5pt}
\begin{equation}
\label{eq:primal_constrained}
    \begin{split}
        \min_{\theta \in \mathbb{R}^{n(1 + m) \times d}} \ & \sum_{i=1}^n \bigg[ \ \ \sum_{j=1}^m f_{i,j}(\theta^{(i,j)}) + \frac{\sigma_i}{2} \| \theta^{(i)} \|^2 \bigg] \\
        \mbox{ such that } & {\theta^{(i)} = \theta^{(j)} \text{ if } j \in \mathcal{N}(i)}, 
       \text{ and } {\theta^{(i,j)} = \theta^{(i)} \ \ \forall j \in \{1, .., m\}}.
    \end{split}
\end{equation}

In the rest of the paper, we use letters $k,\ell$ to refer to any nodes in the augmented graph, and letters $i,j$ to specifically refer to a communication node and one of its virtual nodes. More precisely, we denote $(k, \ell)$ the edge between the nodes $k$ and $\ell$ in the augmented graph. Note that $k$ and $\ell$ can be virtual or communication nodes. We denote $e^{(k)}$ the unit vector of $\mathbb{R}^{n(1 + m)}$ corresponding to node $k$, and $e_{k\ell}$ the unit vector of $\mathbb{R}^{E + nm}$ corresponding to edge $(k,\ell)$. To clearly make the distinction between node variables and edge variables, for any matrix on the set of nodes of the augmented graph $x\in \mathbb{R}^{n(1 + m) \times d}$ we write that $x^{(k)} = x^T e^{(k)}$ for $k \in \{1, ..., n(1+m)\}$ (superscript notation) and for any matrix on the set of edges of the augmented graph $\lambda \in \mathbb{R}^{(E + nm) \times d}$ we write that $\lambda_{k\ell} = \lambda^T e_{k\ell}$ (subscript notation) for any edge $(k, \ell)$. For node variables, we use the subscript notation with a $t$ to denote time, for instance in Algorithm~\ref{algo:sc_adfs}. By a slight abuse of notations, we use indices $(i,j)$ instead of $(k,\ell)$ when specifically refering to virtual edges (or virtual nodes) and denote $\lambda_{ij}$ instead of $\lambda_{i,(i,j)}$ the virtual edge between node $i$ and node $(i,j)$ in the augmented graph. The constraints of Problem~\eqref{eq:primal_constrained} can be rewritten $A^T \theta = 0$ in matrix form, where $A \in \mathbb{R}^{n(1 + m) \times (nm + E)}$ is such that $Ae_{k\ell} = \mu_{k\ell}(e^{(k)} - e^{(\ell)})$ for some $\mu_{k\ell} > 0$. Then, the dual formulation of this problem writes:
\begin{equation}
    \max_{\lambda \in \mathbb{R}^{(nm + E) \times d}} - \sum_{i=1}^n \bigg[\sum_{j = 1}^m f_{i,j}^*\left((A\lambda)^{(i,j)}\right) + \frac{1}{2\sigma_i} \| (A\lambda)^{(i)} \|^2 \bigg],
\end{equation}
where the parameter $\lambda$ is the Lagrange multiplier associated with the constraints of Problem~\eqref{eq:primal_constrained}---more precisely, for an edge $(k,\ell)$, $\lambda_{k\ell} \in \mathbb{R}^d$ is the Lagrange multiplier associated with the constraint $\mu_{k\ell}( e^{(k)} - e^{(\ell)})^T\theta=0$. At this point, the functions $f_{i,j}$ are only assumed to be convex (and not necessarily strongly convex) meaning that the functions $f_{i,j}^*$ are potentially non-smooth. This problem could be bypassed by transferring some of the quadratic penalty from the communication nodes to the virtual nodes before going to the dual formulation. Yet, this approach fails when $m$ is large because the smoothness parameter of $f_{i,j}^*$ would scale as $m / \sigma_i$ at best, whereas a smoothness of order $1 / \sigma_i$ is required to match optimal finite-sum methods. 
A better option is to consider the $f_{i,j}^*$ terms as non-smooth and perform proximal updates on them. The rate of proximal gradient methods such as APCG~\citep{lin2015accelerated} does not depend on the strong convexity parameter of the non-smooth functions $f_{i,j}^*$. Each $f_{i,j}^*$ is $(1 / L_{i,j})$-strongly convex (because $f_{i,j}$ was $(L_{i,j})$-smooth), so we can rewrite the previous equation in order to transfer all the strong convexity to the communication node. Noting that $(A\lambda)^{(i,j)} = - \mu_{ij} \lambda_{ij}$ when node $(i,j)$ is a virtual node associated with node $i$, we rewrite the dual problem as: 
\begin{equation}
\label{eq:dual_problem}
    \min_{\lambda \in \mathbb{R}^{(E + nm) \times d}} q_A(\lambda) +  \sum_{i=1}^n \sum_{j=1}^m \tilde{f^*_{i,j}}(\lambda_{ij}),
\end{equation}
with $\tilde{f^*_{i,j}}: x \mapsto f_{i,j}^*(- \mu_{ij} x) - \frac{\mu_{ij}^2}{2L_{i,j}}\|x\|^2$
and $q_A:x \mapsto {\rm Trace}\big(\frac{1}{2}x^T A^T\Sigma^{-1} A x\big)$, where $\Sigma$ is the diagonal matrix such that ${e^{(i)}}^T\Sigma e^{(i)} = \sigma_i$ if $i$ is a center node and ${e^{(i,j)}}^T\Sigma e^{(i,j)} = L_{i,j}$ if it is the virtual node $(i,j)$. Since dual variables are associated with edges, using coordinate descent algorithms on dual formulations from a well-chosen augmented graph of constraints allows us to handle both computations and communications in the same framework. Indeed, choosing a variable corresponding to an actual edge of the network results in a communication along this edge, whereas choosing a virtual edge results in a local computation step. Then, we balance the ratio between communications and computations by simply adjusting the probability of picking a given kind of edges.

\section{The Algorithm: ADFS Iterations and Expected Error}
\label{sec:alg}

In this section, we detail our new ADFS algorithm. In order to obtain it, we introduce a generalized version of the APCG algorithm~\citep{lin2015accelerated} which we detail in Appendix~\ref{app:generalized_apcg}. Then we apply it to Problem~\eqref{eq:dual_problem} to  get Algorithm~\ref{algo:sc_adfs}. Due to lack of space, we only present the smooth version of ADFS here, but a non-smooth version is presented in Appendix~\ref{app:algo_derivations}, along with the derivations required to obtain Algorithm~\ref{algo:sc_adfs} and Theorem~\ref{thm:rate_adfs}. We denote $A^\dagger$ the pseudo inverse of $A$ and $W_{k\ell} \in \mathbb{R}^{n(1 + m) \times n(1 + m)}$ the matrix such that $W_{k\ell} = (e^{(k)} - e^{(\ell)})(e^{(k)} - e^{(\ell)})^T$ for any edge $(k, \ell)$. Note that variables $x_t$, $y_t$ and $v_t$ from Algorithm~\ref{algo:sc_adfs} are variables associated with the nodes of the augmented graph and are therefore matrices in $\mathbb{R}^{n(1 + m) \times d}$ (one row for each node). They are obtained by multiplying the dual variables of the proximal coordinate gradient algorithm applied to the dual problem of Equation~\eqref{eq:dual_problem} by $A$ on the left. We denote $\sigma_A = \lambda_{\min}^+(A^T\Sigma^{-1}A)$ the smallest non-zero eigenvalue of the matrix $A^T\Sigma^{-1}A$. 

\begin{algorithm}
\caption{ADFS$\left(A, (\sigma_i), (L_{i,j}), (\mu_{k\ell}), (p_{k\ell}), \rho \right)$}
\label{algo:sc_adfs}
\begin{algorithmic}[1]
\STATE $\sigma_A = \lambda_{\min}^+(A^T\Sigma^{-1}A)$, $\tilde{\eta}_{k\ell} = \frac{\rho \mu_{k\ell}^2}{\sigma_A p_{k\ell}}$, $R_{k\ell} = e_{k\ell}^T A^\dagger A e_{k\ell}$ \COMMENT{Initialization}
\STATE $x_0 = y_0 = v_0 = z_0 = 0^{(n + nm) \times d}$ 
\FOR[Run for $K$ iterations]{$t=0$ to $K-1$}
\STATE $y_t = \frac{1}{1 + \rho}\left(x_t + \rho v_t\right)$
\STATE Sample edge $(k,\ell)$ with probability $p_{k\ell}$ \COMMENT{Edge sampled from the augmented graph}
\STATE $z_{t+1} = v_{t+1} = (1 - \rho) v_t + \rho y_t - \tilde{\eta}_{k\ell} W_{k\ell}\Sigma^{-1}y_t$ \COMMENT{Nodes $k$ and $\ell$ communicate $y_t$}
\IF{$(k,\ell)$ is the virtual edge between node $i$ and virtual node $(i,j)$}
\STATE $v_{t+1}^{(i,j)} = {\rm prox}_{\tilde{\eta}_{ij} \tilde{f}^*_{i,j}}\left(z_{t+1}^{(i,j)}\right)$  \COMMENT{Virtual node update using $f_{i,j}$}
\STATE $v_{t+1}^{(i)} = z_{t+1}^{(i)} + z_{t+1}^{(i,j)} - v_{t+1}^{(i,j)}$ \COMMENT{Center node update}

\ENDIF
\STATE $x_{t+1} = y_t + \frac{\rho R_{k\ell}}{p_{k\ell}}(v_{t+1} - (1 - \rho) v_t - \rho y_t)$
\ENDFOR
\STATE \textbf{return} $\theta_K = \Sigma^{-1}v_K$ \COMMENT{Return primal parameter}
\end{algorithmic}
\end{algorithm}

\begin{theorem}
\label{thm:rate_adfs}
We denote $\theta^\star$ the minimizer of the primal function $F:x\mapsto \sum_{i=1}^n f_i(x)$ and $\theta^\star_A$ a minimizer of the dual function $F^*_A = q_A + \psi$. Then $\theta_t$ as output by Algorithm~\ref{algo:sc_adfs} verifies:
\begin{equation}
\label{eq:S}
    \esp{\|\theta_t - \theta^\star\|^2}
    \leq C_0 (1 - \rho)^t, \ \ \ \text{ if } \ \ \ \rho^2 \leq \min_{k\ell} \frac{\lambda_{\min}^+ (A^T \Sigma^{-1} A )}{\Sigma_{kk}^{-1} + \Sigma_{\ell \ell}^{-1}} \frac{p_{k\ell}^2}{\mu_{k\ell}^2 R_{k\ell}},
\end{equation}
with $C_0 = \lambda_{\max}(A^T\Sigma^{-2}A) \left[\|A^\dagger A\theta^\star_A\|^2 + 2 \sigma_A^{-1} \left(F^*_A(0) - F^*_A(\theta^\star_A)\right)\right]  $.
\end{theorem}

We discuss several aspects related to the implementation of Algorithm~\ref{algo:sc_adfs} below, and provide its Python implementation in supplementary material.

\paragraph{Convergence rate.} The parameter $\rho$ controls the convergence rate of ADFS. It is defined by the minimum of the individual rates for each edge, which explicitly depend on parameters related to the functions themselves ($1 / (\Sigma_{kk}^{-1} + \Sigma_{\ell\ell}^{-1}) )$, to the graph topology ($R_{k\ell} = e_{k\ell}^TA^\dagger A e_{k\ell}$), to a mix of both ($\lambda_{\min}^+ (A^T \Sigma^{-1} A ) / \mu_{k\ell}^2$) and to the sampling probabilities of the edges ($p_{k\ell}^2$). Note that these quantities are very different depending on whether edges are virtual or not. In Section~\ref{sec:perfs}, we carefully choose the free parameters $\mu_{k\ell}$ and $p_{k\ell}$ to get the best convergence speed.

\paragraph{Sparse updates.}
Although the updates of Algorithm~\ref{algo:sc_adfs} involve all nodes of the network, it is actually possible to implement them efficiently so that only two nodes are actually involved in each update, as described below. Indeed, $W_{k\ell}$ is a very sparse matrix so $\left(W_{k\ell} \Sigma^{-1}y_t\right)^{(k)} = \mu_{k\ell}^2 (\Sigma_k^{-1} y_t^{(k)} - \Sigma_\ell^{-1} y_t^{(\ell)}) = - \left(W_{k\ell} \Sigma^{-1}y_t\right)^{(\ell)}$ and $\left(W_{k\ell} \Sigma^{-1}y_t\right)^{(h)} = 0$ for $h \neq k,\ell$. Therefore, only the following situations can happen:
\begin{enumerate}
    \item \textbf{Communication updates}: If $(k,\ell)$ is a communication edge, the update only requires nodes $k$ and $\ell$ to exchange parameters and perform a weighted difference between them.
    \item \textbf{Local updates}: If $(k,\ell)$ is the virtual edge between node $i$ and its $j$-th virtual node,
    parameters exchange of line 4 is local, and the proximal term involves function $f_{i,j}$ only.\vspace{-3pt}
    \item \textbf{Convex combinations}: If we choose $h \neq k,\ell$ then $v_{t+1}^{(h)}$ and $y_{t+1}^{(h)}$ are obtained by convex combinations of $y_t^{(h)}$ and $v_t^{(h)}$ so the update is cheap and local. Besides, nodes actually need the value of their parameters only when they perform updates of type 1 or 2. Therefore, they can simply store how many updates of this type they should have done and perform them all at once before each communication or local update.  
\end{enumerate}

\paragraph{Primal proximal step.}
Algorithm~\ref{algo:sc_adfs} uses proximal steps performed on $\tilde{f}^*_{i,j}: x \rightarrow f_{i,j}^*(-\mu_{i,j}x) - \frac{\mu_{ij}^2}{2L_{i,j}}\|x\|^2$ instead of $f_{i,j}$. Yet, it is possible to use Moreau identity to express ${\rm prox}_{\eta \tilde{f}^*_{i,j}}$ using only the proximal operator of $f_{i,j}$, which can easily be evaluated for many objective functions. Note however that this may require choosing a smaller value for $\rho$. The exact derivations are presented in Appendix~\ref{app:algo_derivations_primal_prox}.

\paragraph{Linear case.}
For many standard machine learning problems, $f_{i,j}(\theta) = \ell(X_{i,j}^T\theta)$ with $X_{i,j} \in \mathbb{R}^d$. This implies that $f_{i,j}^*(\theta) = + \infty$ whenever $\theta \notin {\rm Vec}\left(X_{i,j}\right)$. Therefore, the proximal steps on the Fenchel conjugate only have support on $X_{i,j}$, meaning that they are one-dimensional problems that can be solved in constant time using for example the Newton method when no analytical solution is available. Warm starts (initializing on the previous solution) can also be used for solving the local problems even faster so that in the end, a one-dimensional proximal update is only a constant time slower than a gradient update. Note that this also allows to store parameters $v_t$ and $y_t$ as  scalar coefficients for virtual nodes, thus greatly reducing the memory footprint of ADFS.

\section{Distributed Execution and Synchronization Time}
\label{sec:synch_time}
Theorem~\ref{thm:rate_adfs} gives bounds on the expected error after a given number of iterations. To assess the actual speed of the algorithm, it is still required to know how long executing a given number of iterations takes. This is easy with synchronous algorithms such as MSDA or DSBA, in which all nodes iteratively perform local updates or communication rounds. In this case, executing $n_{\rm comp}$ computing rounds and $n_{\rm comm}$ communication rounds simply takes time $n_{\rm comp} + \tau n_{\rm comm}$. ADFS relies on randomized pairwise communications, so it is necessary to sample a \emph{schedule}, \emph{i.e.}, a random sequence of edges from the augmented graph, and evaluate how fast this schedule can be executed. Note that the execution time crucially depends on how many edges can be updated in parallel, which itself depends on the graph and on the random schedule sampled.

\begin{figure}[!htb]
    \centering
        \centering
          \includegraphics[width=\linewidth]{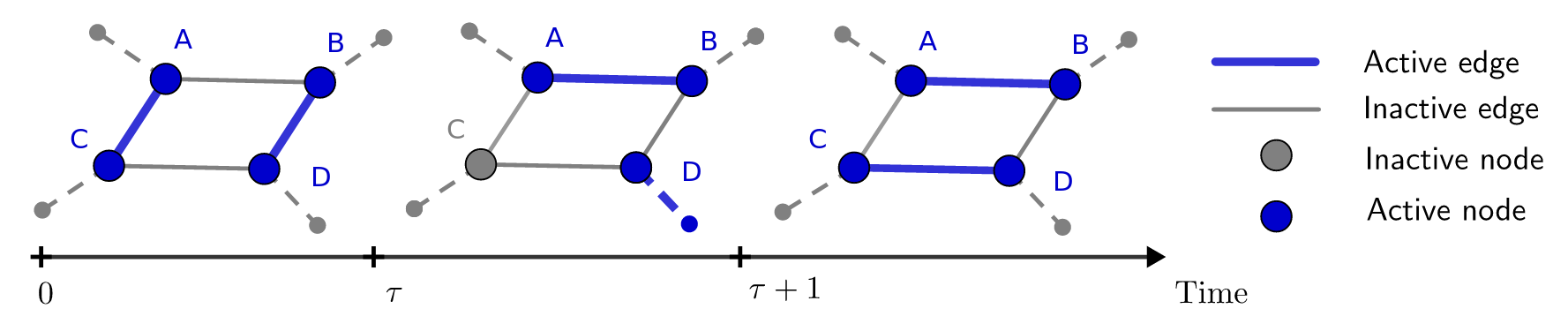}
  \vspace*{-.2cm}
\caption{Illustration of parallel execution and local synchrony. Nodes from a toy graph execute the schedule $[(A,C)$, $(B,D)$, $(A,B)$, $(D)$, $(C,D)]$, where $(D)$ means that node $D$ performs a local update. Each node needs to execute its updates in the partial order defined by the schedule. In particular, node $C$ has to perform update $(A,C)$ and then update $(C,D)$, so it is idle between times $\tau$ and $\tau + 1$ because it needs to wait for node $D$ to finish its local update before the communication update $(C,D)$ can start. We assume $\tau > 1$ since the local update terminates before the communication update $(A,B)$. Contrary to synchronous algorithms, no global notion of rounds exist and some nodes (such as node $D$) perform more updates than others.}
\label{fig:local_synchrony}
\end{figure}

\paragraph{Shared schedule.} Even though they only actively take part in a small fraction of the updates, all nodes need to execute the same schedule to correctly implement Algorithm~\ref{algo:sc_adfs}. To generate this shared schedule, all nodes are given a seed and the sampling probabilities of all edges. This allows them to avoid deadlocks and to precisely know how many convex combinations to perform between $v_t$ and $y_t$.

\paragraph{Execution time.}
The problem of bounding the probability that a random schedule of fixed length exceeds a given execution time can be cast in the framework of fork-join queuing networks with blocking~\citep{zeng2018throughput}. In particular, queuing theory~\citep{baccelli1992synchronization} tells us that the average time per iteration exists for any fixed probability distribution over a given augmented graph. Unfortunately, existing quantitative results are not precise enough for our purpose so we generalize the method introduced by Hendrikx et al.~\citep{hendrikx2018accelerated} to get a finer bound. While their result is valid when the only possible operation is communicating with a neighbor, we extend it to the case in which nodes can also perform local computations. For the rest of this paper, we denote $p_{\rm comm}$ the probability of performing a communication update and $p_{\rm comp}$ the probability of performing a local update. They are such that $p_{\rm comp} + p_{\rm comm} = 1$. We also define $p^{\max}_{\rm comm} = n \max_k \sum_{\ell \in \mathcal{N}(k)}p_{k\ell} / 2$, where neighbors are in the communication network only. When all nodes have the same probability to participate in an update, $p^{\max}_{\rm comm} = p_{\rm comm}$. Then, the following theorem holds (see proof in Appendix~\ref{appendix:average_time}):
\begin{theorem}
\label{thm:synchronization_cost}
Let $T(t)$ be the time needed for the system to execute a schedule of size $t$, \emph{i.e.}, $t$ iterations of Algorithm~\ref{algo:sc_adfs}. If all nodes perform local computations with probability $p_{\rm comp} / n$ with $p_{\rm comp} > p^{\max}_{\rm comm}$ or if $\tau > 1$ then there exists $C < 24$ such that:\vspace{-3pt}
\begin{equation}
    \mathbb{P}\bigg( \frac{1}{t}T(t) \leq \frac{C}{n}\big(p_{\rm comp} + 2 \tau p^{\max}_{\rm comm}\big) \bigg) \rightarrow 1 \mbox{ as } t \rightarrow \infty
\end{equation}
\end{theorem}
Note that the constant $C$ is a worst-case estimate and that it is much smaller for homogeneous communication probabilities. This novel result states that the number of iterations that Algorithm~\ref{algo:sc_adfs} can perform per unit of time increases linearly with the size of the network. This is possible because each iteration only involves two nodes so many iterations can be done in parallel. The assumption $p_{\rm comp} > p_{\rm comm}$ is responsible for the $1 + \tau$ factor instead of $\tau$ in Table~\ref{fig:table_speeds}, which prevents \adfs~from benefiting from \emph{network acceleration} when communications are cheap ($\tau < 1$). Note that this is an actual restriction of following a schedule, as detailed in Appendix~\ref{appendix:average_time}. Yet, network operations generally suffer from communication protocols overhead whereas computing a single proximal update often either has a closed-form solution or is a simple one-dimensional problem in the linear case. Therefore, assuming $\tau > 1$ is not very restrictive in the finite-sum setting.

\section{Performances and Parameters Choice in the Homogeneous Setting}
\label{sec:perfs}
We now prove the time to convergence of \adfs~presented in Table~\ref{fig:table_speeds}, and detail the conditions under which it holds. Indeed, Section~\ref{sec:alg} presents \adfs~in full generality but the different parameters have to be chosen carefully to reach optimal speed. In particular, we have to choose the coefficients $\mu$ to make sure that the graph augmentation trick does not cause the smallest positive eigenvalue of $A^T \Sigma^{-1} A$ to shrink too much. Similarly, $\rho$ is defined in Equation~\eqref{eq:S} by a minimum over all edges of a given quantity. This quantity heavily depends on whether the edge is an actual communication edge or a virtual edge. One can trade $p_{\rm comp}$ for $p_{\rm comm}$ so that the minimum is the same for both kind of edges, but Theorem~\ref{thm:synchronization_cost} tells us that this is only possible as long as $p_{\rm comp} > p_{\rm comm}$. More specifically, we define $L = A_{\rm comm} A_{\rm comm}^T \in \mathbb{R}^{n \times n}$ the Laplacian of the communication graph, with $A_{\rm comm} \in \mathbb{R}^{n \times E}$ such that $A_{\rm comm} e_{k\ell} = \mu_{k\ell}(e^{(k)} - e^{(\ell)})$ for all edge $(k, \ell) \in E^{\rm comm}$, the set of communication edges. Then, we define $\tilde{\gamma} = \min_{(k,\ell) \in E^{\rm comm}} \lambda_{\min}^+(L)n^2 / (\mu_{k\ell}^2 R_{k\ell} E^2)$. As shown in Appendix~\ref{app:gamma_tilde}, $\tilde{\gamma} \approx \gamma$ for regular graphs such as the complete graph or the grid, justifying the use of $\gamma$ instead of $\tilde{\gamma}$ in Table~\ref{fig:table_speeds}. We assume for simplicity that $\sigma_i = \sigma$ and that $\kappa_i = 1 + \sigma_i^{-1}\sum_{j = 1}^m L_{i,j} = \kappa_s$ for all nodes. For virtual edges, we choose $\mu_{ij}^2 = \lambda_{\min}^+(L) L_{i,j} / (\sigma \kappa_i)$ and $p_{ij} = p_{\rm comp} (1 + L_{i,j} \sigma_i^{-1})^{\frac{1}{2}} / (n S_{\rm comp})$ with $S_{\rm comp} = n^{-1}\sum_{i=1}^n\sum_{j=1}^m (1 + L_{i,j} \sigma_i^{-1})^{\frac{1}{2}}$. For communications edges $(k,\ell) \in E^{\rm comm}$, we choose $p_{k\ell} = p_{\rm comm} / E$ and $\mu_{k\ell}^2 = 1 / 2$.

\begin{theorem}
\label{thm:adfs_speed}
If we choose $p_{\rm comm} = \min\Big( 1/2, \Big(1 + S_{\rm comp}\sqrt{\tilde{\gamma}/\kappa_s}\Big)^{-1}\Big)$. Then, running Algorithm~\ref{algo:sc_adfs} for $K_\varepsilon = \rho^{-1} \log(\varepsilon^{-1})$ iterations guarantees $\mathbb{E}\left[\|\theta_{K_\varepsilon} - \theta^\star\|^2 \right] \leq C_0 \varepsilon$, and takes time $T(K_\varepsilon)$, with:
\begin{equation*}
    T(K_\varepsilon) \leq \sqrt{2}C\bigg(m + \sqrt{m\kappa_s} + \sqrt{2}\bigg(1 + 4 \tau\bigg)\sqrt{\frac{\kappa_s}{\tilde{\gamma}}}  \bigg) \log\big(1 / \varepsilon\big)
\end{equation*}
with probability tending to $1$ as $\rho^{-1}\log(\varepsilon^{-1}) \rightarrow \infty$, with the same $C_0$ and $C<24$ as in Theorems~\ref{thm:rate_adfs} and \ref{thm:synchronization_cost}.
\end{theorem}

Theorem~\ref{thm:adfs_speed} assumes that all communication probabilities and condition numbers are exactly equal in order to ease reading. A more detailed version with rates for more heterogeneous settings can be found in Appendix~\ref{app:algo_perfs}. Note that while algorithms such as MSDA required to use polynomials of the initial gossip matrix to model several consecutive communication steps, we can more directly tune the amount of communication and computation steps simply by adjusting $p_{\rm comp}$ and $p_{\rm comm}$.

\section{Experiments}
\label{sec:experiments}

\begin{figure*}
\centering
\begin{subfigure}{.33\textwidth}
  \centering
  \includegraphics[width=1.05\linewidth]{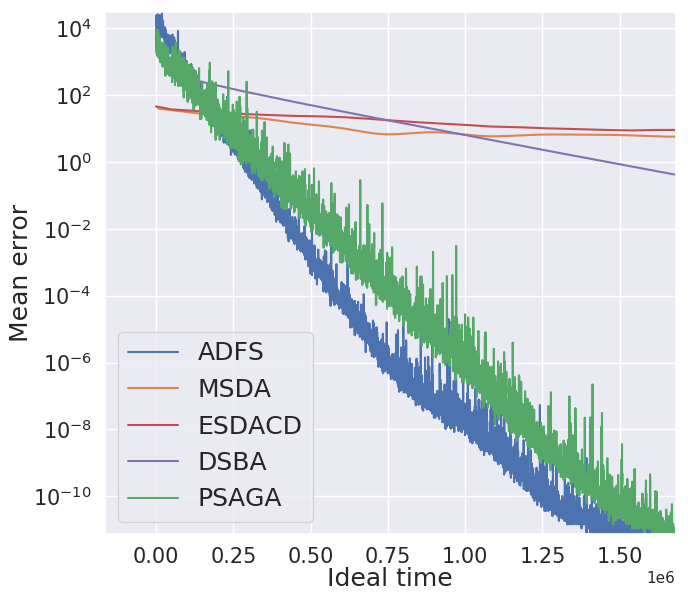}
  \vspace{-15pt}
\caption{Higgs, $n=4$}
\label{fig:2x2_tau1_m1000}
\end{subfigure}%
\begin{subfigure}{.33\textwidth}
  \centering
  \includegraphics[width=1.05\linewidth]{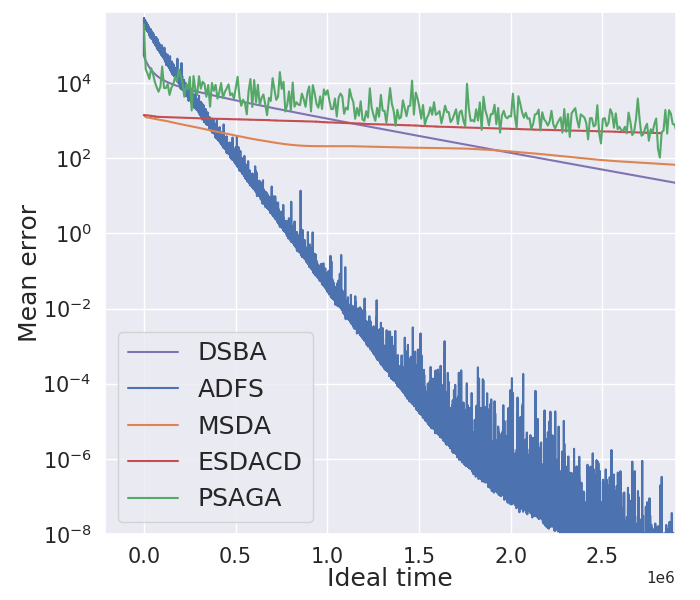}
  \vspace{-15pt}
  \caption{Higgs, $n=100$}
\label{fig:10x10_m300}
\end{subfigure}
\begin{subfigure}{.33\textwidth}
  \centering
  \includegraphics[width=1.05\linewidth]{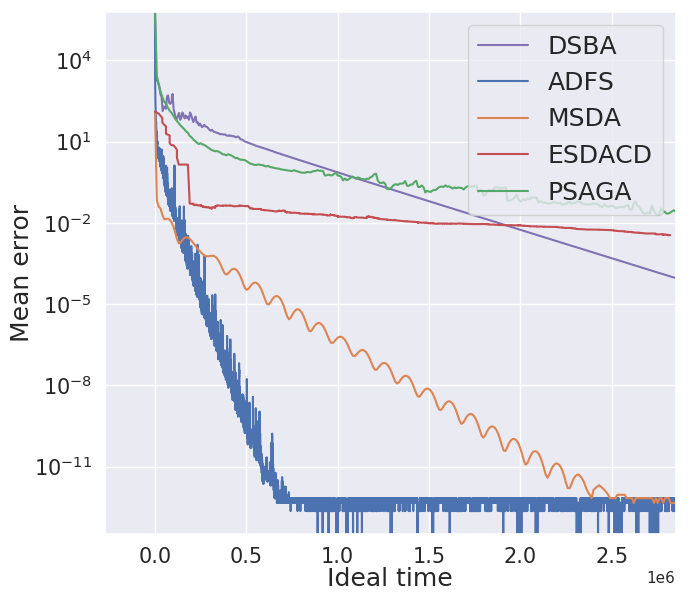}
  \vspace{-15pt}
  \caption{Covtype, $n=100$}
\label{fig:7x7_tau5_m10000}
\end{subfigure}
\vspace{5pt}
\caption{Performances of various decentralized algorithms on the logistic regression task with $m = 10^4$ points per node, regularization parameter $\sigma=1$ and communication delays $\tau=5$ on 2D grid networks of different sizes.}
\label{fig:test}
\end{figure*}
In this section, we illustrate the theoretical results by showing how ADFS compares with MSDA~\citep{scaman2017optimal}, ESDACD~\citep{hendrikx2018accelerated}, Point-SAGA~\citep{defazio2016simple}, and DSBA~\citep{shen2018towards}. All algorithms (except for DSBA, for which we fine-tuned the step-size) were run with out-of-the-box hyperparameters given by theory on data extracted from the standard Higgs and Covtype datasets from LibSVM. The underlying graph is assumed to be a 2D grid network. Experiments were run in a distributed manner on an actual computing cluster. Yet, plots are shown for \emph{idealized times} in order to abstract implementation details as well as ensure that reported timings were not impacted by the cluster status. All the details of the experimental setup as well as a comparison with centralized algorithms can be found in Appendix~\ref{app:experimental_setting}. An implementation of ADFS is also available in supplementary material.

Figure~\ref{fig:2x2_tau1_m1000} shows that, as predicted by theory, ADFS and Point-SAGA have similar rates on small networks. In this case, ADFS uses more computing power but has a small overhead. Figures~\ref{fig:10x10_m300} and~\ref{fig:7x7_tau5_m10000} use a much larger grid to evaluate how these algorithms scale. In this setting, Point-SAGA is the slowest algorithm since it has 100 times less computing power available. MSDA performs quite well on the Covtype dataset thanks to its very good network scaling. Yet, the $m\sqrt{\kappa}$ factor in its rate makes it scales poorly with the condition number $\kappa$, which explains why it struggles on the Higgs dataset. DSBA is slow as well despite the fine-tuning because it is the only non-accelerated method, and it has to communicate after each proximal step, thus having to wait for a time $\tau=5$ at each step. ESDACD does not perform well either because $m > \tau$ and it has to perform as many batch computing steps as communication steps. ADFS does not suffer from any of these drawbacks and therefore outperforms other approaches by a large margin on these experiments. This illustrates the fact that ADFS combines the strengths of accelerated stochastic algorithms, such as Point-SAGA, and fast decentralized algorithms, such as MSDA. 

\section{Conclusion}
In this paper, we provided a novel accelerated stochastic algorithm for decentralized optimization. To the best of our knowledge, it is the first decentralized algorithm that successfully leverages the finite-sum structure of the objective functions to match the rates of the best known sequential algorithms while having the network scaling of optimal batch algorithms. The analysis in this paper could be extended to better handle heterogeneous settings, both in terms of hardware (computing times, delays) and local functions (different regularities). Finally, finding a locally synchronous algorithm that can take advantage of arbitrarily low communication delays (beyond the $\tau > 1$ limit) to scale to large graphs is still an open problem.

\section*{Acknowledgement}
We acknowledge support from the European Research Council (grant SEQUOIA 724063). 

\bibliographystyle{plainnat}
\bibliography{biblio}

\newpage

\appendix 

Section~\ref{app:generalized_apcg} is a self-contained section with the statement and proofs of the extended APCG algorithm. Then, Section~\ref{app:algo_derivations} presents the derivations required to obtain ADFS from the extended APCG algorithm. Section~\ref{appendix:average_time} is dedicated to the study of waiting time in the locally synchronous model, and the analysis of the speed of ADFS for a specific choice of parameters is then given in Section~\ref{app:algo_perfs}. Finally, Section~\ref{app:experimental_setting} details the experimental setting and gives additional experiments involving centralized algorithms.

\section{Generalized APCG}
\label{app:generalized_apcg}
In this section, we study the generic problem of accelerated proximal coordinate descent. We give an algorithm that works with \emph{arbitrary sampling} of the coordinates, thus yielding a stronger result than state-of-the-art approaches~\citep{lin2015accelerated, fercoq2015accelerated}. This is a key contribution that allows to obtain fast rates when sampling probabilities are heterogeneous and determined by the problem. It is especially useful in our case to pick different probabilities for computing and for communicating. We also extend the result to the case in which the function is strongly convex only on a subspace. Since this section of the Appendix is intended to detail the extended APCG general algorithm for a generic problem, it is mostly self-contained, and notations are in particular different from the rest of the paper. More specifically, we study the following generic problem:
\begin{equation}
\label{eq:generic_problem}
    \min_{x \in \mathbb{R}^d}\ \  f_A(x) + \sum_{i=1}^d \psi_i(x^{(i)}),
\end{equation}
where all the functions $\psi_i$ are convex and $f_A$ is such that there exists a matrix $A$ such that $f_A$ is $(\sigma_A)$-strongly convex on ${\rm Ker}(A)^\perp$, the orthogonal of the kernel of $A$. Since, $A^\dagger A$ is the projector on ${\rm Ker}(A)^\perp$, (recall that $A^\dagger$ is the pseudo-inverse of $A$), the strong convexity on this subspace can be written as the fact that for all $x,y \in \mathbb{R}^d$:
\begin{equation}
\label{eq:sc_semi_norm}
f_A(x) - f_A(y) \! \geq \! \nabla f_A(y)^T \! A^\dagger A(x - y) + \textstyle \frac{\sigma_A}{2}(x - y)^T \! A^\dagger A (x - y).    
\end{equation}
Note that this implies that $f_A$ is constant on ${\rm Ker}(A)$, so in particular there exists a function $f$ such that for any $x \in \mathbb{R}^d$, $f_A(x) = f(Ax)$. In this case, $\sigma_A$ is such that $x^TA^T\nabla^2f(y) Ax \geq \sigma_A\|x\|^2$ for any $x\in {\rm Ker}(A)^\perp$ and $y\in \mathbb{R}^d$. Besides, $f_A$ is assumed to be $(M_i)$-smooth in direction $i$ meaning that its gradient in the direction $i$ (noted $\nabla_i f_A$) is $(M_i)$-Lipschitz. This is the general setting of the problem of Equation~\eqref{eq:dual_problem}, that can be recovered by taking $f_A=q_A$ and $\psi_i = \tilde{f_i^*}$. Proximal coordinate gradient algorithms are known to work well for these problems, which is why we would like to use APCG~\citep{lin2014accelerated}. Yet, we would like to pick different probabilities for computing and communication edges, whereas APCG only handles uniform coordinates sampling. Furthermore, the first term is strongly-convex only on the orthogonal of the kernel of the matrix $A$, so APCG cannot be applied straightforwardly. Therefore, we introduce an extended version of APCG, presented in Algorithm~\ref{algo:generalized_apcg}, and we explicit its rate in Theorem~\ref{thm:gen_apcg}. This extended APCG can then directly be applied to solve the problem of Equation~\eqref{eq:dual_problem}.

\subsection{Algorithm and results}
In this appendix, since there is no need to distinguish between primal and dual variables variables as in the main text, we denote $e_i \in \mathbb{R}^d$ the unit vector corresponding to coordinate $i$, and $x^{(i)} = e_i^Tx$ for any $x \in \R^d$. Let $R_i = e_i^T A^\dagger A e_i$ and $p_i$ be the probability that coordinate~$i$ is picked to be updated. Constant $S$ is such that $S^2 \geq \frac{M_i R_i}{p_i^2}$ for all $i$. Then, following the approaches of Nesterov and Stich~\citep{nesterov2017efficiency} and Lin et al.~\citep{lin2015accelerated}, we fix $A_0, B_0 \in \mathbb{R}$ and recursively define sequences $\alpha_t, \beta_t, a_t, A_t$ and $B_t$ such that:
\begin{align*}
    & a_{t+1}^2 S^2 = A_{t+1} B_{t+1}, & B_{t+1} = B_t + \sigma_A a_{t+1}, & \ \ \ \ \ \ \ \ \ A_{t+1} = A_t + a_{t+1}, \\
    & \alpha_t = \frac{a_{t+1}}{A_{t+1}}, & \beta_t = \frac{\sigma_A a_{t+1}}{B_{t+1}}. & 
\end{align*} 
Finally, we introduce the sequences $(y_t)$, $(v_t)$ and $(x_t)$, that are all initialized at $0$, and $(w_t)$ such that for all $t$, $w_t = (1 - \beta_t) v_t + \beta_t y_t$. We define $\eta_{i,t} = \frac{a_{t+1}}{B_{t+1} p_i}$ and the proximal operator:\vspace{-5pt}
\begin{equation*}
    {\rm prox}_{\eta_{i,t} \psi_i}: x \mapsto \arg\min_v \frac{1}{2\eta_{i,t}}\|v - x\|^2 + \psi_i(v).
\end{equation*}
\begin{algorithm}
\caption{Generalized APCG$(A_0, B_0, S, \sigma_A)$}
\label{algo:generalized_apcg}
\begin{algorithmic}
\STATE $y_0 = 0$, $v_0 = 0$, $t = 0$
\WHILE{$t < T$}
\STATE $y_t = \frac{(1 - \alpha_t) x_t + \alpha_t(1 - \beta_t)v_t}{1 - \alpha_t \beta_t}$
\STATE Sample $i$ with probability $p_i$
\STATE $v_{t+1} = z_{t+1} = (1 - \beta_t) v_t + \beta_t y_t - \eta_{i,t} \nabla_i f_A(y_t)$
\STATE $v_{t+1}^{(i)} = {\rm prox}_{\eta_i \psi_i}\left(z_{t+1}^{(i)}\right)$
\STATE $x_{t+1} = y_t + \frac{\alpha_t R_i}{p_i}(v_{t+1} - (1 - \beta_t) v_t - \beta_t y_t)$
\ENDWHILE
\end{algorithmic}
\end{algorithm}
We denote $\nabla_i f_A = e_i e_i^T \nabla f_A$ the coordinate gradient of $f_A$ along direction $i$. For generalized APCG to work well, the proximal operator needs to be taken in the subspace defined by the projector $A^\dagger A$, and so the non-smooth $\psi_i$ terms have to be separable after composition with $A^\dagger A$. Since $A^\dagger A$ is a projector, this constraint is equivalent to stating that either $R_i = 1$ (projection does not affect the coordinate $i$), or $\psi_i = 0$ (no proximal update to make).
\begin{assumption}
\label{assumption:gen_apcg}
The functions $f_A$ and $\psi$ are such that equation Equation~\eqref{eq:sc_semi_norm} holds for some $\sigma_A \geq 0$ and for all $i \in \mathbb{R}^d$, $f_A$ is $(M_i)$-smooth in direction $i$ and $\psi$ and $A$ are such that either $R_i = 1$ or $\psi_i = 0$.
\end{assumption}
This natural assumption allows us to formulate the proximal update in standard squared norm since the proximal operator is only used for coordinates $i$ for which $A^\dagger A e_i = e_i$. Then, we formulate Algorithm~\ref{algo:generalized_apcg} and analyze its rate in Theorem~\ref{thm:gen_apcg}.
\begin{theorem}
\label{thm:gen_apcg}
Let $F: x \mapsto f_A(x) + \sum_{i=1}^d \psi_i\left(x^{(i)}\right)$ such that Assumption~\ref{assumption:gen_apcg} holds. If $S$ is such that $S^2 \geq \frac{M_i R_i}{p_i^2}$ and $1 - \beta_t - \frac{\alpha_t R_i}{p_i} \geq 0$, the sequences $v_t$ and $x_t$ generated by APCG verify:
\begin{equation*}
    B_t\esp{\|v_t - \theta^\star\|^2_{A^\dagger A}} + 2 A_t \left[\esp{F(x_t)} - F(\theta^\star)\right] \leq C_0,
\end{equation*}
where $C_0 = B_0 \|v_0 - \theta^\star\|^2 + 2A_0\left[F(x_0) - F(\theta^\star)\right]$ and $\theta^\star$ is a minimizer of $F$. The rate of APCG depends on $S$ through the sequences $\alpha_t$ and $\beta_t$. 
\end{theorem}

Our extended APCG algorithm is also closely related with an arbitrary sampling version of APPROX~\cite{fercoq2015accelerated}. Yet, APPROX has an explicit formulation with a more flexible block selection rule than choosing only one coordinate at a time. Similarly to Lee and Sidford~\cite{lee2013efficient}, it also uses iterations that can be more efficient, especially in the linear case. These extensions can also be applied to APCG under the same assumptions, but this is beyond the scope of this paper. Theorem~\ref{thm:gen_apcg} is a general method that in particular requires to set values for $A_0$, $B_0$, $\alpha_0$ and $\beta_0$. The two following corollaries give choices of parameters depending on whether $\sigma_A > 0$ or $\sigma_A = 0$, along with the rate of APCG in these cases.

\begin{corollary}[Strongly Convex case]
\label{corr:sc_apcg}
Let $F$ be such that it verifies the assumptions of Theorem~\ref{thm:gen_apcg}. If $\sigma_A > 0$, we can choose for all $t\in \mathbb{N}$ $\alpha_t = \beta_t = \rho$ and $A_t = \sigma_A^{-1} B_t = (1 - \rho)^{-t}$ with $\rho = \sqrt{\sigma_A}S^{-1}$. In this case, the condition $1 - \beta_t - \frac{\alpha_t R_i}{p_i} \geq 0$ can be weakened to $1 - \frac{\alpha_t R_i}{p_i} \geq 0$ and it is automatically satisfied by our choice of $S$, $\alpha_t$ and $\beta_t$. In this case, APCG converges linearly with rate $\rho$, as shown by the following result: $$\sigma_A \esp{\|v_t - \theta^\star\|^2_{A^\dagger A}} + 2\left[\esp{F(x_t)} - F(\theta^\star)\right] \leq C_0 (1- \rho)^t$$
\end{corollary}

\begin{corollary}[Convex case]
\label{corr:cvx_apcg}
Let $F$ be such that it verifies the assumptions of Theorem~\ref{thm:gen_apcg}. If $\sigma_A = 0$, we can choose $\beta_t = 0$, $B_0 = 1$ and $A_0 = \frac{3B_0}{S^2 p_R^2}$ with $p_R = \min_i \frac{p_i}{R_i}$. In this case, the condition $1 - \beta_t - \frac{\alpha_t R_i}{p_i} \geq 0$ is always satisfied for our choice of $S$ and the error verifies:
$$ \esp{F(x_t)} - F(\theta^\star) \leq \frac{2}{t^2} \left[S^2 r_t^2 + \frac{6}{p_{R}^2}\left[F(x_0) - F(\theta^\star)\right]\right],$$
with $r_t^2 = \|v_0 - \theta^\star\|^2_{A^\dagger A} -  \mathbb{E}[\|v_t - \theta^\star\|^2_{A^\dagger A}]$.
\end{corollary}
In the convex case, we only have control over the objective function $F$ and not over the parameters. This in particular means that it is only possible to have guarantees on the dual objective in the case of non-smooth ADFS. 

\subsection{Proof of Theorem~\ref{thm:gen_apcg}}
Before starting the proof, we define $w_t = (1 - \beta_t) v_t + \beta_t y_t$, and:
$$V_i^t(v) = \frac{B_{t+1}p_i}{2 a_{t+1}} \|v - w_t^{(i)} + \eta_i e_i^T \nabla f(y_t)\|^2 + \psi_i(v).$$ Then, we give the following lemma, that we prove later:

\begin{lemma}
\label{lemma:lyapunov_psi}
If either $1 - \beta_t - \frac{\alpha_t}{p_i} \geq 0$ or $\alpha_t = \beta_t$ and $1 - \frac{\alpha_t}{p_i} \geq 0$ for any $i$ such that $\psi_i \neq 0$, then for any $t$ and $i$ such that $\psi_i \neq 0$, we can write $x_t^{(i)} = \sum_{l=0}^t \delta^{(i)}_t(l) v_l^{(i)}$ such that $\sum_{l=0}^t \delta^{(i)}_t(l) = 1$ and for any $l$, $\delta^{(i)}_t(l) \geq 0$. We define $\hat{\psi}_t^{(i)} = \sum_{l=0}^t \delta^{(i)}_t(l) \psi_i(v_l^{(i)})$ and $\hat{\psi}_t = \sum_{i=1}^d \hat{\psi}_t^{(i)}$. Then, if $R_i = 1$ whenever $\psi_i \neq 0$, $\psi(x_t) \leq \hat{\psi}_t$ and:
\begin{equation}
    \mathbb{E}_{i_t}\left[\hat{\psi}_{t+1}\right] \leq \alpha_t \psi(\tilde{v}_{t+1}) + (1 - \alpha_t)\hat{\psi}_t.
\end{equation}
where $\tilde{v}_{t+1}^{(i)} = \arg \min_v V_i^t(v)$ for all $i$. In particular, $v_{t+1}^{(i_t)} = \tilde{v}_{t+1}^{(i_t)}$ and $v_{t+1}^{(j)} = w_t^{(j)}$ for $j \neq i_t$.
\end{lemma}

Note that Lemma~\ref{lemma:lyapunov_psi} is a generalization to arbitrary sampling probabilities of the beginning of the proof in~\citep{lin2015accelerated}. We can now prove the main theorem.

\begin{proof}[Proof of Theorem~\ref{thm:gen_apcg}]
This proof follows the same general structure as Nesterov and Stich ~\citep{nesterov2017efficiency}. In particular, it follows from expanding the $\|v_{t+1} - \theta^\star\|^2$ term. In the original proof, $v_{t+1} = w_t - g$ where $g$ is a gradient term so the expansion is rather straightforward. In our case, $v_{t+1}$ is defined by a proximal mapping so a bit more work is required. Yet, similar terms will appear, plus the function values of the non-smooth term that we control with Lemma~\ref{lemma:lyapunov_psi}. We start by showing the following equality:
\begin{align}
\label{eq:sc_prox}
\begin{split}
    \frac{B_{t+1}p_i}{2a_{t+1}} [\|v_{t+1} - \theta^\star\|^2_{A^\dagger A} + &\|v_{t+1} - w_t\|^2_{A^\dagger A} - \|\theta^\star - w_t\|^2_{A^\dagger A}] \\
    &\leq \langle \nabla_i f_A(y_t), \theta^\star - v_{t+1} \rangle_{A^\dagger A} + \psi_i\left({\theta^\star}^{(i)}\right) - \psi_i\left(v_{t+1}^{(i)}\right).
    \end{split}
\end{align}
When $\psi_i = 0$, it follows from using $v_{t+1} = w_t - \frac{a_{t+1}}{B_{t+1} p_i}\nabla_i f_A(y_t)$ and basic algebra (expanding the squared terms).

When $\psi_i \neq 0$, $A^\dagger Ae_i = e_i$ because $e_i^TA^\dagger Ae_i = 1$ and $A^\dagger A$ is a projector. Therefore, we obtain 
\begin{equation}
\label{eq:vt1proj}
    \|v_{t+1} - \theta^\star\|^2_{A^\dagger A} - \|w_t - \theta^\star\|^2_{A^\dagger A} = \|v_{t+1}^{(i)} - {\theta^\star}^{(i)}\|^2 - \|w_t^{(i)} - {\theta^\star}^{(i)}\|^2,
\end{equation}
because $v_{t+1}$ is equal to $w_t$ for coordinates other than $i$. We now use the strong convexity of $V_i^t$ at points $v_{t+1}^{(i)}$ (its minimizer, by definition) and ${\theta^\star}^{(i)}$ ($i$-th coordinate of a minimizer of $F$) to write that $V_i^t(v_{t+1}^{(i)}) + \frac{B_{t+1}p_i}{2a_{t+1}}\|v_{t+1}^{(i)} - {\theta^\star}^{(i)}\|^2 \leq V_i^t({\theta^\star}^{(i)})$. This is a key step from the proof of Lin et al.~\citep{lin2015accelerated}. Then, expanding the $V_i^t$ terms yields:
\begin{align*}
    \frac{B_{t+1}p_i}{2a_{t+1}} &\left[\|v_{t+1}^{(i)} - {\theta^\star}^{(i)} \|^2 + \|v_{t+1}^{(i)} - w_t^{(i)} + \frac{a_{t+1}}{B_{t+1} p_i}\nabla_i f_A(y_t)\|^2 - \|\theta^\star - w_t + \frac{a_{t+1}}{B_{t+1} p_i}\nabla_i f_A(y_t)\|^2 \right]\\
    &\leq  \psi_i\left({\theta^\star}^{(i)}\right) - \psi_i\left(v_{t+1}^{(i)}\right).
\end{align*}
We can now retrieve Equation~\eqref{eq:sc_prox} by pulling gradient terms out of the squares and using Equation~\eqref{eq:vt1proj}. We now evaluate each term of Equation~\eqref{eq:sc_prox}. First of all, we use the form of $x_{t+1}$ and the fact that $w_t - v_{t+1} = e_i^T(w_t - v_{t+1})$ (only one coordinate is updated) to show:
\begin{align*}
    \esp{\frac{a_{t+1}}{p_i}\langle \nabla_i f_A(y_t), \theta^\star - v_{t+1} \rangle_{A^\dagger A}} &= a_{t+1} \esp{\langle \frac{1}{p_i}\nabla_i f_A(y_t), \theta^\star - w_t\rangle_{A^\dagger A}}\\
    &+ A_{t+1}\esp{\langle \nabla_i f_A(y_t), \frac{\alpha_t}{p_i}(w_t - v_{t+1}) \rangle_{A^\dagger A}} \\
    &= a_{t+1} \langle \nabla f_A(y_t), \theta^\star - w_t \rangle_{A^\dagger A} + A_{t+1}\esp{\langle \nabla_i f_A(y_t), y_t - x_{t+1} \rangle},
\end{align*}
where we used that $R_i = e_i^TA^\dagger Ae_i$ and $y_t - x_{t+1} = \frac{\alpha_t R_i}{p_i}(w_t - v_{t+1})$. 

The rest of this proof closely follows the analysis from Hendrikx et al.~\citep{hendrikx2018accelerated}, which is an adaptation of Nesterov and Stich~\citep{nesterov2017efficiency} to strong convexity only on a subspace. The main difference is that it is also necessary to control the function values of $\psi$, which is done using Lemma~\ref{lemma:lyapunov_psi}. For the first term, we use the strong convexity of $f$ as well as the fact that $w_t = y_t - \frac{1 - \alpha_t}{\alpha_t} (x_t - y_t)$ to obtain:
\begin{align*}
a_{t+1}& \nabla f_A(y_t)^T A^\dagger A (\theta^\star - w_t) = a_{t+1} \nabla f_A(y_t)^T A^\dagger A \left(\theta^\star - y_t + \frac{1 - \alpha_t}{\alpha_t}(x_t - y_t)\right)\\
&\leq a_{t+1} \left(f_A(\theta^\star) - f_A(y_t) - \frac{1}{2} \sigma_A \|y_t - \theta^\star\|^2_{A^\dagger A} + \frac{1 - \alpha_t}{\alpha_t}(f_A(x_t) - f_A(y_t)) \right) \\
&\leq a_{t+1} f_A(\theta^\star) - A_{t+1} f_A(y_t) + A_t f_A(x_t) - \frac{1}{2} a_{t+1} \sigma_A \|y_t - \theta^\star\|^2_{A^\dagger A}.
\end{align*}
For the second term, we use the fact that $x_{t+1} - y_t$ has support on $e_i$ only (just like $v_{t+1} - w_t$) and the directional smoothness of $f_A$ to obtain:
\begin{align*}
    A_{t+1} \langle \nabla_i f_A(y_t), y_t - x_{t+1} \rangle &\leq A_{t+1} \left[f_A(y_t) - f_A(x_{t+1}) + \frac{M_i}{2}\|x_{t+1} - y_t\|^2\right] \\
    &\leq A_{t+1} \left(f_A(y_t) - f_A(x_{t+1})\right) +\frac{B_{t+1}}{2} \frac{M_i R_i}{p_i^2} \frac{a_{t+1}^2}{A_{t+1}B_{t+1}} R_i\|e_i^T (v_{t+1} - w_t)\|^2\\
    &\leq A_{t+1} \left(f_A(y_t) - f_A(x_{t+1})\right) +\frac{B_{t+1}}{2} \|v_{t+1} - w_t\|^2_{A^\dagger A}.
\end{align*}
Noting $\Delta f_A(x_t) = \esp{f(x_t)} - f_A(\theta^\star)$ and remarking that $a_{t+1} = A_{t+1} - A_t$, we obtain, using that $\alpha_t = \frac{a_{t+1}}{A_{t+1}}$: 
\begin{align*}
    \esp{\frac{a_{t+1}}{p_i}\langle \nabla_i f_A(y_t), \theta^\star - v_{t+1} \rangle_{A^\dagger A}} \leq A_t \Delta f_A(x_t)& - A_{t+1} \Delta f_A(x_{t+1}) + \frac{B_{t+1}}{2}\esp{ \|w_t - v_{t+1}\|^2_{A^\dagger A}}\\
    &- \frac{a_{t+1}\sigma_A}{2}\|y_t - \theta^\star\|^2_{A^\dagger A}.
\end{align*}
Using Lemma~\ref{lemma:lyapunov_psi}, we derive in the same way:
\begin{align*}
    \esp{\frac{a_{t+1}}{p_i} \left[\psi_i\left({\theta^\star}^{(i)}\right) - \psi_i\left(v_{t+1}^{(i)}\right)\right]} &= a_{t+1}\psi(\theta^\star) - A_{t+1} \alpha_t  \psi(\tilde{v}_{t+1})\\
    &\leq A_t \left(\hat{\psi}_t - \psi(\theta^\star)\right) - A_{t+1}\left(\hat{\psi}_{t+1} - \psi(\theta^\star) \right).
\end{align*}
Now, we can multiply Equation~\eqref{eq:sc_prox} by $\frac{a_{t+1}}{p_i}$ and take the expectation over $i$. The $\|v_{t+1} - w_t\|^2_{A^\dagger A}$ terms cancel and we obtain:
\begin{align*}
\frac{B_{t+1}}{2}\esp{\|v_{t+1} - \theta^\star\|^2_{A^\dagger A}} + A_{t+1} \Delta \hat{F}_A(x_{t+1}) \leq A_t \Delta \hat{F}_A(x_t) + \frac{B_{t+1}}{2}\|w_t - \theta^\star\|^2_{A^\dagger A} - \frac{a_{t+1}\sigma_A}{2}\|y_t - \theta^\star\|^2_{A^\dagger A},
\end{align*}
where $\Delta \hat{F}_A(x_t) = \Delta f_A(x_t) +  \esp{\hat{\psi}_t} - \psi(\theta^\star)$. Convexity of the squared norm yields $\|w_t - \theta^\star\|^2_{A^\dagger A} \leq (1 - \beta_t)\|v_t - \theta^\star\|^2_{A^\dagger A} + \beta_t \|y_t - \theta^\star\|^2_{A^\dagger A}$. Now remarking that $B_{t+1}(1 - \beta_t) = B_t$ and $a_{t+1}\sigma_A = B_{t+1} \beta_t$, and summing the inequalities until $t=0$, we obtain:
\begin{align*}
B_{t}\|v_{t} - \theta^\star\|^2_{A^\dagger A} + 2A_{t} \Delta \hat{F}_A(x_{t}) \leq 2A_0 \Delta F_A(x_0) + B_0 \|v_0 - \theta^\star\|^2_{A^\dagger A}.
\end{align*}
We finish the proof by using the fact that $\psi(x_t) \leq \hat{\psi}_t$ and $\psi(x_0) = \hat{\psi}_0$ since $x_0 = v_0$.
\end{proof} 

Now that we have proven Theorem~\ref{thm:gen_apcg}, we can proceed to the proof of Lemma~\ref{lemma:lyapunov_psi}.

\begin{proof}[Proof of Lemma~\ref{lemma:lyapunov_psi}]
This lemma is a generalization of the lemma from APCG with arbitrary probabilities (instead of uniform ones). It still uses the fact that $x_t$ can be written as a convex combination of $(v_l)_{l \leq t}$, but it requires to use a different convex combination for each coordinate of $x_t$, thus crucially exploiting the separability of the proximal term. If coordinate $i$ is such that $\psi_i = 0$, then $\hat \psi^{(i)}_{t+1} \leq \alpha_t \psi_i(\tilde v_{t+1}^{(i)}) + (1 - \alpha_t) \hat \psi^{(i)}_{t}$ is automatically satisfied for any $\delta_t^{(i)}$. For coordinates $i$ such that $\psi_i \neq 0$ (and so $R_i = 1$), we start by expressing $x_{t+1}$ in terms of $x_t$, $v_{t+1}$ and $v_t$ . More precisely, we write that for any $t > 0$:
\begin{equation*}
     x_{t+1}^{(i)} = y_t^{(i)} + \frac{\alpha_t}{p_i} (v_{t+1}^{(i)} - w_t^{(i)}).
\end{equation*}
Indeed, either coordinate $i$ is updated at time $t$ or $v_{t+1}^{(i)} = w_t^{(i)}$ so the previous equation always holds. We can then develop the $w_t$ and $y_t$ terms to obtain $x_{t+1}^{(i)}$ only in function of $x_t^{(i)}$, $v_t^{(i)}$ and $v_{t+1}^{(i)}$:
\begin{align*}
    x_{t+1}^{(i)} 
    &= \frac{\alpha_t}{p_i}v_{t+1}^{(i)} + \left(1 - \frac{\alpha_t \beta_t}{p_i}\right) y_t^{(i)}  - \frac{\alpha_t(1 - \beta_t)}{p_i}v_{t}^{(i)}\\
    &= \frac{\alpha_t}{p_i}v_{t+1}^{(i)} + \left(1 - \frac{\alpha_t \beta_t }{p_i}\right)\frac{(1 - \alpha_t) x_t^{(i)} + \alpha_t(1 - \beta_t)v_t^{(i)}}{1 - \alpha_t \beta_t}  - \frac{\alpha_t(1 - \beta_t)}{p_i}v_{t}^{(i)}\\
    &= \frac{\alpha_t }{p_i}v_{t+1}^{(i)} + \alpha_t(1 - \beta_t)\left[\frac{1 - \frac{\alpha_t \beta_t }{p_i}}{1 - \alpha_t \beta_t} - \frac{1}{p_i} \right]v_t^{(i)} +  \left(1 - \frac{\alpha_t \beta_t }{p_i}\right)\frac{(1 - \alpha_t)}{1 - \alpha_t \beta_t} x_t^{(i)}\\
    &= \frac{\alpha_t }{p_i}v_{t+1}^{(i)} + \frac{\alpha_t(1 - \beta_t)}{1 - \alpha_t \beta_t}\left(1 - \frac{1}{p_i}\right)v_t^{(i)} + \left(1 - \frac{\alpha_t \beta_t }{p_i}\right)\frac{(1 - \alpha_t)}{1 - \alpha_t \beta_t} x_t^{(i)} .
\end{align*}
At this point, all coefficients sum to 1. Indeed, they all sum to 1 at the first line and we have expressed $w_t^{(i)}$ and then $y_t^{(i)}$ as convex combinations of other terms, thus keeping the value of the sum unchanged. Yet, $p_i < 1$ so the coefficient on the second term is negative. Fortunately, it is possible to show that the $v_t^{(i)}$ term in the decomposition of $x_t^{(i)}$ is large enough so that the $v_t^{(i)}$ term in the decomposition of $x_{t+1}^{(i)}$ is positive. More precisely, we now show by recursion that for $t \geq 0$:
\begin{equation}
\label{eq:xt_convex}
    x_{t+1}^{(i)} = \frac{\alpha_t }{p_i}v_{t+1}^{(i)} + \sum_{l=0}^{t} \delta^{(i)}_{t+1}(l) v_l^{(i)},
\end{equation}
with $\delta^{(i)}_{t+1}(l) \geq 0$ for $l \leq t$. For $t=0$, $x_0 = v_0$ and $x_1^{(i)} = \frac{\alpha_0 }{p_i}v_1^{(i)} + \left( 1 - \frac{\alpha_0 }{p_i}\right)v_0^{(i)}$. We now assume that Equation~\eqref{eq:xt_convex} holds for a given $t > 0$, and expand $\delta_{t+1}^{(i)}(t)$ to show that it is positive. Using that $\delta_{t}^{(i)}(t) = \frac{\alpha_t}{p_i}$, we write:
\begin{align*} 
    \delta_{t+1}^{(i)}(t) = &\frac{\alpha_t(1 - \beta_t)}{1 - \alpha_t \beta_t}\left(1 - \frac{1}{p_i}\right) + \frac{\alpha_t }{p_i} \left(1 - \frac{\alpha_t \beta_t }{p_i}\right)\frac{(1 - \alpha_t)}{1 - \alpha_t \beta_t} \\
    &= \frac{\alpha_t}{1 - \alpha_t \beta_t}\left[(1 - \beta_t) \left(1 - \frac{1}{p_i}\right) + \frac{(1 - \alpha_t) }{p_i}\left(1 - \frac{\alpha_t \beta_t }{p_i}\right)\right]\\
    &= \frac{\alpha_t}{1 - \alpha_t \beta_t}\left[1 - \beta_t - \frac{1}{p_i} + \frac{\beta_t }{p_i} + \frac{1}{p_i} - \frac{\alpha_t }{p_i} - (1 - \alpha_t)\frac{\alpha_t \beta_t}{p_i^2}\right]\\
    &= \frac{\alpha_t}{1 - \alpha_t \beta_t}\left[\left(1 - \beta_t - \frac{\alpha_t }{p_i}\right) + \frac{\beta_t }{p_i}\left(1 - (1 - \alpha_t)\frac{\alpha_t }{p_i}\right)\right].
\end{align*}
We conclude that $\delta_{t+1}^{(i)}(t) \geq 0$ since $1 - \beta_t - \frac{\alpha_t }{p_i} \geq 0$. Note that this condition can be weakened to $1 - \frac{\alpha_t^2 }{p_i^2} \geq 0$ when $\beta_t = \alpha_t$ or when $\beta_t = 0$. We also deduce from the form of $x_{t+1}^{(i)}$ that for $l < t$, the only coefficients on $v_l^{(i)}$ in the development of $x_{t+1}^{(i)}$ come from the $x_t^{(i)}$ term and so:
\begin{equation}
    \delta_{t+1}^{(i)}(l) = \left(1 - \frac{\alpha_t \beta_t }{p_i}\right)\frac{(1 - \alpha_t)}{1 - \alpha_t \beta_t} \delta_t^{(i)}(l),
\end{equation}
so these coefficients are positive as well. Since they also sum to $1$, it implies that $x_t^{(i)}$ is a convex combination of the $v_l^{(i)}$ for $l \leq t$, and we use the convexity of $\psi_i$ to write: $$\psi_i(x_t^{(i)}) = \psi_i\left(\sum_{l=0}^t \delta_t^{(i)}(l) v_l^{(i)}\right) \leq \sum_{l=0}^t \delta_t^{(i)}(l) \psi_i(v_l^{(i)}) =  \hat{\psi}^{(i)}_t.$$ 

Now, we can properly express $\hat{\psi}^{(i)}_{t+1}$ using the decomposition of $x_{t+1}^{(i)}$ in terms of $\delta_{t+1}^{(i)}$:
\begin{align*}
    \esp{\hat{\psi}^{(i)}_{t+1}} &= \esp{\frac{\alpha_t }{p_i}\psi_i(v_{t+1}^{(i)})} + \frac{\alpha_t(1 - \beta_t)}{1 - \alpha_t \beta_t}\left(1 - \frac{1}{p_i}\right)\psi_i(v_t^{(i)}) + \left(1 - \frac{\alpha_t \beta_t }{p_i}\right)\frac{1 - \alpha_t}{1 - \alpha_t \beta_t}\sum_{l=0}^t \delta_t^{(i)}(l) \psi_i(v_l^{(i)})\\
    &= \alpha_t \psi_i(\tilde v_{t+1}^{(i)}) + \left(1 - p_i\right)\frac{\alpha_t}{p_i}\psi_i(w_{t}^{(i)})  + \frac{\alpha_t(1 - \beta_t)}{1 - \alpha_t \beta_t}\left(1 - \frac{1}{p_i}\right)\psi_i(v_t^{(i)}) + \left(1 - \frac{\alpha_t \beta_t }{p_i}\right)\frac{1 - \alpha_t}{1 - \alpha_t \beta_t}\hat{\psi}^{(i)}_{t}\\
\end{align*}
At this point, we use the convexity of $\psi_i$ to develop $\psi_i(w_{t}^{(i)})$ and then $\psi_i(y_{t}^{(i)})$ in the following way:
\begin{align*}
     \psi_i(w_{t}^{(i)}) &\leq (1 - \beta_t)\psi_i(v_{t}^{(i)}) + \beta_t \psi_i(y_{t}^{(i)}) \\
     &\leq (1 - \beta_t)\psi_i(v_{t}^{(i)}) + \frac{\beta_t}{1 - \alpha_t \beta_t}\left[(1 - \alpha_t) \psi_i(x_{t}^{(i)}) + \alpha_t(1 - \beta_t) \psi_i(v_{t}^{(i)})\right]\\
     &= \frac{1 - \beta_t}{1 - \alpha_t \beta_t}\psi_i(v_{t}^{(i)}) + \frac{\beta_t(1 - \alpha_t)}{1 - \alpha_t \beta_t} \psi_i(x_{t}^{(i)}).
\end{align*}

If we plug these expressions into the development of $\esp{\hat{\psi}^{(i)}_{t+1}}$, the $\psi_i(v_{t}^{(i)})$ terms cancel and we obtain:

\begin{align*}
    \esp{\hat{\psi}^{(i)}_{t+1}} &\leq \alpha_t \psi_i(\tilde v_{t+1}^{(i)}) + \alpha_t\left(\frac{1}{p_i} - 1\right)\frac{\beta_t(1 - \alpha_t)}{1 - \alpha_t \beta_t} \psi_i(x_{t}^{(i)}) + \left(1 - \frac{\alpha_t \beta_t }{p_i}\right)\frac{1 - \alpha_t}{1 - \alpha_t \beta_t}\hat{\psi}^{(i)}_{t}\\
\end{align*}

We now use the fact that $\psi_i(x_{t}^{(i)}) \leq \hat{\psi}^{(i)}_{t}$ (by convexity of $\psi_i$) to get:

\begin{align*}
    \esp{\hat{\psi}^{(i)}_{t+1}} &\leq \alpha_t \psi_i(\tilde v_{t+1}^{(i)}) + \frac{1 - \alpha_t}{1 - \alpha_t \beta_t}\left[\alpha_t \beta_t \left(\frac{1}{p_i} - 1\right) + \left(1 - \frac{\alpha_t \beta_t }{p_i}\right)\right]\hat{\psi}^{(i)}_{t}\\
    &\leq \alpha_t \psi_i(\tilde v_{t+1}^{(i)}) + (1 - \alpha_t)\hat{\psi}^{(i)}_{t}
\end{align*}

This holds for any coordinate $i$ and so $\esp{\hat{\psi}_{t+1}} \leq \alpha_t \psi(\tilde v_{t+1} + (1 - \alpha_t)\hat{\psi}_{t}$ for all $t \geq 0$, which finishes the proof of the lemma.
\end{proof}

\subsection{Proof of the corollaries}
Now that that we have proven the main result, we show how specific choices of parameters lead to fast algorithms. 

\begin{proof}[Proof of Corollary~\ref{corr:sc_apcg}]
If $\sigma_A > 0$, then the parameters can be chosen as $\alpha_t = \beta_t = \rho = \frac{\sqrt{\sigma_A}}{S}$, with $A_t = (1 - \rho)^{-t}$ and $B_t = \sigma_A A_t$. These expressions can then be plugged into the recursion to verify that they do satisfy it. This choice is classic and slightly suboptimal for small values of $t$ compared with the choice made by Nesterov and Stich~\citep{nesterov2017efficiency}.

Yet, it remains to prove that $\alpha_t R_i \leq p_i$ to verify the assumptions of Theorem~\ref{thm:gen_apcg}. This assumption was directly verified in the case of APCG thanks to the uniform probabilities. We show that this also holds in the arbitrary-sampling formulation with strong convexity on a subspace, and this result validates our choice of parameters. In particular, we write: 
$$\alpha_t R_i = \rho R_i \leq \frac{\sqrt{\sigma_A}}{S} R_i \leq \sqrt{\frac{\sigma_A R_i}{M_i}} p_i.$$
Then, we take $x^\star$ such that $\nabla f_A(x^\star) = 0$ and use the smoothness and $(A^\dagger A)$-strong convexity of $f_A$  to write that for any coordinate $i$ and $h > 0$:
$$\frac{M_i}{2}\|h e_i\|^2 \geq f(x^\star + h e_i) - f(x^\star) \geq \frac{\sigma_A}{2}\|h e_i\|_{A^\dagger A}.$$

In particular, this means that $M_i \geq \sigma_A R_i$, which means that $\alpha_t R_i \leq p_i$ for all $i$.
\end{proof}

\begin{proof}[Proof of Corollary~\ref{corr:cvx_apcg}]
If $\sigma_A = 0$ then we have to choose $\beta_t = 0$ for all $t$. Then, we can choose $B_t = B_0$ for a any $B_0 > 0$ and a direct recursion yields: $$A_{t+1} = A_t + \frac{B_0}{2S^2}\left(1 + \sqrt{1 + 4S^2 B_0^{-1} A_t}\right).$$
This in particular shows that $(A_t)$ is an increasing sequence. Coefficients $(a_t)$ can be computed using
$$a_{t+1} = A_{t+1} - A_t = \frac{B_0}{2S^2}\left(1 + \sqrt{1 + 4S^2 B_0^{-1} A_t}\right),$$
and so the sequence $(\alpha_t)$ is given by:
$$\alpha_t = \frac{a_{t+1}}{A_{t+1}} = 1 - \frac{A_t}{A_{t+1}} = 1 - \frac{A_t}{A_t + \frac{B_0}{2S^2}\left(1 + \sqrt{1 + 4S^2b^{-1}A_t}\right)}.$$
We would like to choose $A_0 = 0$ but this would lead to $\alpha_0 = 1$ and would not respect $\alpha_t R_i \leq p_i$ for all $i$ so we choose $A_0 > 0$. Since $(A_t)$ is an increasing sequence, $(\alpha_t)$ is a decreasing sequence, and in particular it is enough to verify that $\alpha_0 \leq p_R$ with $p_R = \min_i p_i / R_i$ to have that $\alpha_t R_i \leq p_i$ for all $i$ and all $t$. Therefore, we want $A_0 > 0$, such that:
$$p_R \geq 1 - \frac{1}{1 + \frac{B_0}{2A_0S^2}\left(1 + \sqrt{1 + 4S^2B_0^{-1}A_0}\right)},$$
which can be rewritten $\frac{p_R}{1 - p_R} \geq \frac{1}{X}\left(1 + \sqrt{1 + 2X}\right)$ with $X = 2S^2B_0^{-1}A_0$. If $p_R \geq 1$ then the equation is verified for any value of $A_0$ and we could actually even have chosen $A_0 = 0$. Otherwise, we choose $X \geq 6 / p_R^2 \geq 6$. Then:
$$\frac{p_R}{1 - p_R} \geq p_R \geq \frac{\sqrt{6}}{\sqrt{X}} = \frac{\sqrt{X}(\sqrt{6} - \sqrt{2})}{X} + \frac{1}{X}\sqrt{2X} \geq \frac{2}{X} + \frac{\sqrt{2X}}{X} \geq \frac{1}{X}(1 + \sqrt{1 + 2X}),$$
where we have used that $\sqrt{X}(\sqrt{6} - \sqrt{2}) \geq 2$ for $X \geq 6$ and $1 + \sqrt{2X} \geq \sqrt{1 + 2X}$. In particular, the constraint $\alpha_t R_i \leq p_i$ is satisfied with the choice:
$$A_0 = \frac{3 B_0}{S^2 p_R^2}.$$
Note that the constant $3$ is not tight but allows for an easy expression which always hold. Since $A_0 > 0$, a direct recursion yields $A_t \geq \frac{B_0 t^2}{4S^2}$. We call $r_t^2 = \|v_0 - \theta^\star_A\|^2_{A^\dagger A} - \mathbb{E}[\|v_t - \theta^\star_A\|^2_{A^\dagger A}]$, and $F_t =  \mathbb{E}[F_A(x_t)] - F_A(\theta^\star_A)$, then:

$$F_t \leq \frac{1}{2A_t}\left(B_0r_t^2 + 2A_0F_0\right) = \frac{B_0}{2A_t}\left(r_t^2 + \frac{2}{S^2 p_{\min}^2}F_0\right) \leq \frac{2S^2}{t^2}\left(r_t^2 + \frac{6}{S^2 p_{R}^2}F_0\right),$$ which finishes the proof.

\end{proof}

\section{Algorithm Derivation}
\label{app:algo_derivations}
\subsection{Projection of virtual edges}
Theorem~\ref{app:generalized_apcg} requires that for any coordinate $i$, either the proximal part $\psi_i = 0$ or the coordinate is such that $e_i^TA^\dagger A e_i = 1$, which is equivalent to having $A^\dagger A e_i = e_i$. In our case, $\psi_{k\ell} = 0$ when $(k,\ell)$ is a communication edge. Lemma~\ref{lemma:acrossa} is a small result that shows that the projection condition is satisfied by virtual edges.
\begin{lemma}
\label{lemma:acrossa}
If $(k,\ell)$ is a virtual edge then $R_{k\ell} = 1$.
\end{lemma}
\begin{proof}
Let $x \in \mathbb{R}^{E + nm}$ such that $Ax = 0$. From the definition of $A$, either $x = 0$ or the support of $x$ is a cycle of the graph. Indeed, for any edge $(k,\ell)$, $Ae_{k\ell}$ has non-zero weights only on nodes $k$ and $\ell$. Virtual nodes have degree one, so virtual edges are parts of no cycles and therefore $x^T e_{k,\ell} = 0$ for all virtual edges $(k,\ell)$. Operator $A^\dagger A$ is the projection operator on the orthogonal the kernel of $A$, so it is the identity on virtual edges. 
\end{proof}

\subsection{From edge variables to node variables}
Taking the dual formulation implies that variables are associated with edges rather than nodes. Although it could be possible to work with edge variables, it is generally inefficient. Indeed, the algorithm needs variable $Ay_t$ instead of variable $y_t$ for the gradient computation so standard methods work directly with $Ay_t$ \citep{scaman2017optimal, hendrikx2018accelerated}.

In this section, we call $\tilde v_t$, $\tilde y_t$ and $\tilde z_t$ the dual variable sequences in $\mathbb{R}^{E + nm}$ obtained by applying Algorithm~\ref{algo:generalized_apcg} on the dual problem of Equation~\ref{eq:dual_problem}. The new update equations can be retrieved by multiplying each line of Algorithm~\ref{algo:generalized_apcg} by $A$ on the left, so that for example $v_t = A \tilde v_t$. Yet, there is still a $\tilde z_{t+1}$ term because of the presence of the proximal update. More specifically, we write for the virtual edge between node $i$ and its $j$-th virtual node:
\begin{equation}
    \tilde v_{t+1}^T e_{ij} = {\rm prox}_{\eta_{ij} \psi_{i,j}}\left(\tilde z_{t+1}^T e_{ij}\right).
\end{equation}
Fortunately, this update only modifies $\tilde v_{t+1}$ when $\psi_{i,j} \neq 0$. This means that $z_{t+1}$ is only modified for local computation edges. Since local computation nodes only have one neighbour, the form of $A$ ensures that for any $\tilde z \in \mathbb{R}^{n(1 + m)}$ and virtual edge $(k, \ell)$ corresponding to node $i$ and its $j$-th virtual node, $(A\tilde z)^{(i,j)} = - \mu_{k\ell} \tilde z_{k\ell}$. In particular, if node $k$ is the center node $i$ and node $\ell$ is the virtual node $(i,j)$, the proximal update can be rewritten:
\begin{align*}
    \left(A \tilde v_{t+1}\right)^{(i,j)} &= - \mu_{ij} {\rm prox}_{\eta_{ij} \psi_{i,j}}\left(- \frac{1}{\mu_{ij}} (A \tilde z_{t+1})^{(i,j)}\right)\\
    &= - \mu_{ij} \arg\min_v \frac{1}{2\eta_{ij}}\|v - \left(- \frac{1}{\mu_{ij}} (A \tilde z_{t+1})^{(i,j)}\right)\|^2 + \psi_{i,j}(v)\\
    &= - \mu_{ij} \arg\min_v \frac{1}{2\eta_{ij}\mu_{ij}^2}\|- \mu_{ij}v - (A \tilde z_{t+1})^{(i,j)}\|^2 + f_{i,j}^*(- \mu_{ij} v) - \frac{\mu_{ij}^2}{2L_{i,j}}\|v\|^2\\
    &= \arg\min_{\tilde{v}} \frac{1}{2\eta_{ij}\mu_{ij}^2}\|\tilde{v} - (A \tilde z_{t+1})^{(i,j)}\|^2 + f_{i,j}^*(\tilde{v}) - \frac{1}{2L_{i,j}}\|\tilde{v}\|^2\\
    &= {\rm prox}_{\eta_{ij}\mu_{ij}^2\tilde{f}^*_{i,j}}\left(\left(A \tilde z_{t+1}\right)^{(i,j)}\right),
\end{align*}
where $\tilde{f}^*_{i,j}: x \rightarrow f_{i,j}^*(x) - \frac{1}{2L_{i,j}}\|x\|^2$. For the center node, the update can be written:
\begin{align*}
    \left(A \tilde v_{t+1}\right)^{(i)} &= \left(A \tilde z_{t+1}\right)^{(i)} - \mu_{ij} e_{ij}^T \tilde z_{t+1} + \mu_{ij} {\rm prox}_{\eta_{ij} \psi_{i,j}}\left(- \frac{1}{\mu_{ij}} \left(A \tilde z_{t+1}\right)^{(i,j)}\right)\\
    &= \left(A \tilde z_{t+1}\right)^{(i)} +  \left(A \tilde z_{t+1}\right)^{(i,j)} - {\rm prox}_{\eta_{ij}\mu_{k\ell}^2\tilde{f}^*_{i,j}}\left(\left(A \tilde z_{t+1}\right)^{(i,j)}\right).
\end{align*}

\subsection{Primal proximal updates}
\label{app:algo_derivations_primal_prox}
 Moreau identity~\citep{parikh2014proximal} provides a way to retrieve the proximal operator of $f^*$ using the proximal operator of $f$, but this does not directly apply to $\tilde{f}^*_{i,j}$, making its proximal update hard to compute when no analytical formula is available to compute $\tilde{f}^*_{i,j}$. Fortunately, the proximal operator of $\tilde{f}^*_{i,j}$ can be retrieved from the proximal operator of $f^*_{i,j}$. More specifically, if we denote $\tilde \eta_{ij} = \eta_{ij}\mu_{ij}^2$ (it is clear in this section that they refer to the edge between node $i$ and its virtual node $j$), then we can also express the update only in terms of $f_{i,j}^*$:
\begin{align*}
    {\rm prox}_{\tilde{\eta}_{ij} \tilde{f}^*_{i,j}}\left(\left(A \tilde z_{t+1}\right)^{(i,j)}\right) &= \arg\min_v \frac{1}{2\tilde{\eta}_{ij}} \|v - \left(A\tilde z_{t+1}\right)^{(i,j)}\|^2 + \tilde{f}^*_{i,j}(v) - \frac{1}{2L_{i,j}}\|v\|^2\\
    &= \arg\min_v \frac{1}{2}\left(\tilde{\eta}_{ij}^{-1} - L_{i,j}^{-1}\right)\|v\|^2 - \tilde{\eta}_{ij}^{-1} v^T \left(A \tilde z_{t+1}\right)^{(i,j)} + \tilde{f}^*_{i,j}(v) \\
    &= \arg\min_v \frac{1}{2\left(\tilde{\eta}_{ij}^{-1} - L_{i,j}^{-1}\right)^{-1}}\|v - \left(1 - \tilde{\eta}_{ij} L_{i,j}^{-1}\right)^{-1} \left(A \tilde z_{t+1}\right)^{(i,j)}\|^2 + \tilde{f}^*_{i,j}(v) \\
    &= {\rm prox}_{\left(\tilde{\eta}_{ij}^{-1} - L_{i,j}^{-1}\right)^{-1} f^*_{i,j}}\left(\left(1 - \tilde{\eta}_{ij} L_{i,j}^{-1}\right)^{-1} \left(A \tilde z_{t+1}\right)^{(i,j)}\right).
\end{align*}
Then, we use the identity:
\begin{equation}
    {\rm prox}_{\left(\eta f\right)^*}(x) = \eta {\rm prox}_{\eta^{-1} f^*}\left(\eta^{-1} x\right),
\end{equation}

and the Moreau identity to write that:
\begin{equation}
    {\rm prox}_{\eta f^*}(x) = x -  \eta {\rm prox}_{\eta^{-1} f}\left(\eta^{-1} x\right).
\end{equation}

This allows us to retrieve the proximal operator on $\tilde{f}^*_{i,j}$ using only the proximal operator on $f_{i,j}$:
\begin{equation}
    \left(1 - \tilde{\eta}_{ij} L_{i,j}^{-1}\right) {\rm prox}_{\tilde{\eta}_{ij} \tilde{f}^*_{i,j}}\left(\left(A \tilde z_{t+1}\right)^{(i,j)}\right) = \left(A\tilde z_{t+1}\right)^{(i,j)} - \tilde{\eta}_{ij} {\rm prox}_{\left(\tilde{\eta}_{ij}^{-1} - L_{i,j}^{-1}\right) f} \left(\tilde{\eta}_{ij}^{-1} \left(A\tilde z_{t+1}\right)^{(i,j)}\right).
\end{equation}

Note that the previous calculations are valid as long as  $\tilde{\eta}_{ij} L_{i,j}^{-1} \leq 1$ for all virtual edges. A way to bound this is to replace by the values of $\mu_{ij}^2$ and $\sigma_A$ to get: 
$$\rho \leq \frac{\kappa_i}{2\kappa} p_{ij}.$$
The constraint $\rho < \min_{ij} p_{ij}$ was already enforced by APCG, so this simply gives another constraint that is generally verified unless nodes have very different local objectives (which should not happen if $m$ is big enough). 

\subsection{Smooth case}
If the functions $f_{i,j}$ are smooth then the functions $f_{i,j}^*$ are strongly convex and so function $q_A$ is strongly convex. ADFS can then be obtained by applying Algorithm~\ref{algo:generalized_apcg} to Problem~\eqref{eq:dual_problem}. The value of $S$ is obtained by remarking that $q_A$ is $\mu_{ij}^2\left(\Sigma_i^{-1} + \Sigma_j^{-1}\right)$ smooth in the direction $(i,j)$ and $\lambda_{\min}^+\left(A^T\Sigma^{-1}A\right)$ strongly convex on the orthogonal of the kernel of $A$. Lemma~\ref{lemma:acrossa} guarantees that either $R_i = 1$ (virtual edges) or $\psi_i = 0$ (communication edges), so we can apply Corollary~\ref{corr:sc_apcg} to get:
    \begin{equation*}
    B_t\esp{\|\tilde{v}_t - \theta^\star_A\|^2_{A^\dagger A}} + 2 A_t \left[\esp{F^*_A(A\tilde{x}_t)} - F^*_A(\theta^\star_A)\right] \leq C_0,
    \end{equation*}
where $\tilde{v}_t$ and $\tilde{x}_t$ are the dual variables and $C_0$ is the same as in Theorem~\ref{thm:rate_adfs}. ADFS works with variables $v_t = A\tilde{v}_t$ and $x_t = A\tilde{x}_t$ instead. Then, we use the fact that for any $x$, $F^*_A(x) = F^*_A(A^\dagger Ax)$ to write that $\esp{F^*_A(\tilde{x}_t)} = \esp{F^*_A(A^\dagger x_t)}$. Following Lin et al.~\citep{lin2015accelerated}, and noting $q: x \mapsto \frac{1}{2} x^T \Sigma^{-1} x$ the primal optimal point $\theta^\star$ can be retrieved as $\theta^\star = \nabla q (A\theta^\star_A) = \Sigma^{-1}A\theta^\star_A$, where $\theta^\star_A$ is the optimal dual parameter. Finally, $$\lambda_{\max}(A^T\Sigma^{-2}A)^{-1} \|\theta_t - \theta^\star\|^2 \leq \lambda_{\max}(A^T\Sigma^{-2}A)^{-1} \|\Sigma^{-1}A(\tilde{v}_t - \theta^\star_A)\|^2 \leq \|\tilde{v}_t - \theta^\star_A\|^2_{A^\dagger A},$$ which finishes the proof of Theorem~\ref{thm:rate_adfs}. Note that APCG also gives a guarantee in terms of dual function values at points $x_t$ but we drop it in order to have a simpler statement. 

\subsection{Non-smooth setting}
\label{app:non_smooth_adfs}
Extended APCG can be applied to the problem of Equation~\eqref{eq:dual_problem} even if function $q_A$ is not strongly convex on the orthogonal of the Kernel of $A$. This is for example the case when the functions $f_{i,j}$ are not smooth so that $\Sigma^{-1}$ has diagonal entries equal to 0 and therefore ${\rm Ker}(A^T\Sigma^{-1}A) \not\subset {\rm Ker}(A)$ so $\sigma_A = 0$. In this case, the choice of coefficients from Corollary~\ref{corr:cvx_apcg} leads to Algorithm~\ref{algo:ns_adfs}, a formulation of ADFS that provides error guarantees when primal functions $f_{i,j}$ are not smooth. More formally, if we define $F^*: x \rightarrow \sum_{i=1}^n \bigg[\sum_{j = 1}^m f_{i,j}^*\left(x^{(i,j)}\right) + \frac{1}{2\sigma_i} \| x^{(i)} \|^2 \bigg]$, then:
\begin{theorem}
\label{thm:adfs_non_smooth}
If the functions $f_{i,j}$ are non-smooth then NS-ADFS guarantees:
$$ \esp{F^*(x_t)} - F^*(\theta^\star) \leq \frac{2}{t^2} \left[\frac{S^2}{\lambda_{\min}^+(A^TA)} r_t^2 + \frac{6}{p_{R}^2}\left[F^*(x_0) - F^*(\theta^\star)\right]\right],$$
with $r_t^2 = \|v_0 - \theta^\star\|^2 - \|v_t - \theta^\star\|^2$.
\end{theorem}
The guarantees provided by Theorem~\ref{thm:adfs_non_smooth} are weaker than in the smooth setting. In particular, we lose linear convergence and get the classical  accelerated sublinear $O(1/t^2)$ rate. We also lose the bound on the primal parameters--- recovering primal guarantees is beyond the scope of this work.

\begin{algorithm}
\caption{NS-ADFS}
\label{algo:ns_adfs}
\begin{algorithmic}
\STATE $x_0 = 0$, $v_0 = 0$, $t = 0$, $p_R = \min_i p_i / R_i$, $A_0 = \frac{3 B_0}{S^2 p_R^2}$, $\eta_{ij} = \frac{\mu_{ij}^2}{p_{ij}}$
\WHILE{$t < T$}
\STATE $A_{t+1} = A_t + \frac{1}{2S^2}\left(1 + \sqrt{1 + 4S^2A_t}\right)$
\STATE $a_{t+1} = A_{t+1} - A_t$, $\alpha_t = \frac{a_{t+1}}{A_{t+1}}$
\STATE $y_t = (1 - \alpha_t)x_t + \alpha_t v_t$
\STATE Sample $(i,j)$ with probability $p_{ij}$
\STATE $v_{t+1} = z_{t+1} = v_t - a_{t+1} \eta_{ij} W_{ij}\Sigma^{-1}y_t$
\IF{$(i,j)$ is a computation edge} 
\STATE $v_{t+1}^{(i,j)} = {\rm prox}_{a_{t+1} \eta_{ij} f^*_{i,j}}\left(z_{t+1}^{(i,j)}\right)$
\STATE $v_{t+1}^{(i)} = z_{t+1}^{(i)} + z_{t+1}^{(i,j)} - v_{t+1}^{(i,j)}$
\ENDIF
\STATE $x_{t+1} = y_t + \frac{\alpha_t R_{ij}}{p_{ij}}(v_{t+1} - v_t)$
\ENDWHILE
\STATE \textbf{return} $\theta_t = \Sigma^{-1}v_t$
\end{algorithmic}
\end{algorithm}

Note that the extra $\lambda_{\min}^+(A^TA)$ term comes from the fact that Theorem~\ref{thm:adfs_non_smooth} is formulated with primal parameter sequences $x_t = A \tilde x_t$. Also note that $\alpha_t = \grando{t^{-1}}$, and $\frac{a_{t+1}}{B_{t+1}} = \grando{t}$. The leading constant governing the convergence rate is $\frac{\lambda_{\min}^+\left(A^TA\right)}{S^2}$, which is very related to the constant for the smooth case, simply that the $\Sigma^{-1}$ factor is removed. Therefore, we can obtain in the same way that if we choose $\mu_{ij}^2 = \frac{\lambda_{\min}^+(L)}{1 + m}$ when $(i,j)$ is a computation edge then we get: $$ \lambda_{\min}^+(A^TA) \geq \frac{\lambda_{\min}^+(L)}{2(m + 1)}.$$

Optimizing parameter $\rho$ in order to minimize time yields $\rho_{\rm comp} = \rho_{\rm comm}$ again, now leading in the homogeneous case to choosing: $$p_{\rm comm}^* = \left(1 + \sqrt{\frac{\tilde{\gamma}m^2}{2(1 + m)}}\right)^{-1}.$$

\section{Average Time per Iteration}
\label{appendix:average_time}

\subsection{More communications implies more waiting}
A fundamental assumption for Theorem~\ref{thm:synchronization_cost} is to assume that $p_{\rm comm} < p_{\rm comp}$. In particular, it prevents $p_{\rm comm}$ from being too high since $p_{\rm comm} + p_{\rm comp} = 1$. Although this assumption seems quite restrictive in the first place, it is very intuitive to want to avoid $p_{\rm comm}$ from being too high, especially in the limit of $p_{\rm comm} \rightarrow 1$ and $\tau$ arbitrarily small. Consider that one node (say node $0$) starts a local update at some point. Communications are very fast compared to computations so it is very likely that the neighbors of node $0$ will only perform communication updates, and they will do so until they have to perform one with node $0$. At this point, they will have to wait until node~$0$ finishes its local computation, which can take a long time. Now that the neighbors of node $0$ are also blocked waiting for the computation to finish, their neighbors will start establishing a dependence on them rather quickly. If the probability of computing is small enough and if the computing time is large enough, all nodes will sooner or later need to wait for node $0$ to finish its local update before they can continue with the execution of their part of the schedule. In the end, only node $0$ will actually be performing computations while all the others will be waiting. 

This phenomenon is not restricted to the limit case presented above and the synchronization cost blows up as soon as $p_{\rm comm} > p_{\rm comp}$ and $\tau < 1$. In the proof below, the goal is to bound the total expected weight $\sum_{i=1}^n \mathds{E}\left[X^t(i, w)\right]$ for $w$ higher than a given threshold. Local computing operations will move mass from small values of $w$ to higher values of $w$. On the other hand, communication operations will introduce synchronization between two nodes, thus increasing the total available mass $\sum_{w \geq 0} \sum_{i=1}^n \mathds{E}\left[X^t(i, w)\right]$ (and not just moving it to higher values of $w$) because it will duplicate the mass for $X^t(i, w)$ to $X^t(j,w)$ if nodes $i$ and $j$ communicate. This is the technical reason why $p_{\rm comm} < p_{\rm comp}$ is needed for this proof. 
 
\subsection{Detailed average time per iteration proof}
The goal of this section is to prove Theorem~\ref{thm:synchronization_cost}. The proof is an extension of the proof of Theorem~2 from Hendrikx et al.~\citep{hendrikx2018accelerated}. Similarly, we denote $t$ the number of iterations that the algorithm performs and $\tau_c^{ij}$ the random variable denoting the time taken by a communication on edge $(i,j)$. Similarly, $\tau_l^i$ denotes the time taken by a local computation at node $i$. Then, we introduce the random variable $X^t(i, w)$ such that if edge $(i,j)$ is activated at time $t+1$  (with probability $p_{ij}$), then for all $w \in \mathbb{N}^*$:
\begin{equation*}
    X^{t+1}(i,w) = X^t(i, w - \tau_c^{ij}(t)) + X^t(j, w - \tau_c^{ij}(t)),
\end{equation*}

where $\tau_c^{ij}(t)$ is the realization of $\tau_c^{ij}$ corresponding to the time taken by activating edge $(i,j)$ at time $t$. If node $i$ is chosen for a local computation, which happens with probability $p^{\rm comp}_i$ then $X^{t+1}(i, w + \tau_l^i(t)) = X^t(i, w)$ for all $w$. Otherwise, $X^{t+1}(j, w) = X^t(j, w)$ for all $w$. At time $t=0$, $X^0(i, 0) = 1$ and $X^0(i,w) = 0$ for all $w$. Lemma~\ref{lemma:ptgttheta} gives a bound on the probability that the time taken by the algorithm to complete $t$ iterations is greater than a given value, depending on variables $X^t$. Note that a Lemma similar to the one by Hendrikx et al.~\citep{hendrikx2018accelerated} holds although variable $X$ has been modified. 

\begin{lemma}
\label{lemma:ptgttheta}
We denote $T_{\max}(t)$ the time at which the last node of the system finishes iteration $t$. Then for all $\nu > 0$:

\begin{equation*}
    \mathbb{P}\left(T_{\max}(t) \geq \nu t\right) \leq \sum_{w \geq \nu t} \sum_{i=1}^n \mathds{E}\left[X^t(i, w)\right].
\end{equation*}

\end{lemma}

\begin{proof}
We first prove by induction on $t$ that for any $i \in \{1, .., n\}$: 

\begin{equation}
\label{eq:rec_ti}
    T_i(t) = \max_{w \in \mathbb{N}, X^t(i,w) > 0} w.
\end{equation}

To ease notations, we write $w_{\max}(i,t) = \max_{w \in \mathbb{N}, X^t(i,w) > 0} w$. The property is true for $t=0$ because $T_i(0) = 0$ for all $i$.

We now assume that it is true for some fixed $t > 0$ and we assume that edge $(k,l)$ has been activated at time $t$. For all $i \notin \{k, l\}$, $T_i(t+1) = T_i(t)$ and for all $w \in \mathbb{N}^*$, $X^{t+1}(i, w) = X^t(i, w)$ so the property is true. Besides, if $j\neq l$,

\begin{align*}
    w_{\max}(k, t+1) &= \max_{w \in \mathbb{N^*}, X^t(k,w - \tau_c(t)) + X^t(l,w - \tau_c^{kl}(t)) > 0} w \\
    &= \max_{w \in \mathbb{N}, X^t(i,w) + X^t(i,w) > 0} w + \tau_c^{kl}(t) \\
    &= \tau_c(t) + \max\left(w_{\max}(k, t), w_{\max}(l, t)\right)\\
    &= \tau_c^{kl}(t) + \max \left(T_k(t), T_l( t)\right) = T_k(t+1).
\end{align*}

Similarly if $k=l$ (a local computation is performed at iteration $t$), then $w_{\max}(k, t+1) = \tau_l^k(t) + w_{\max}(k, t) = T_k(t) + \tau_l^k(t) = T_k(t+1)$. Then, we use the union bound and the the fact that having $X^t(i, w) > 0$ is equivalent to having $X^t(i, w) \geq 1$ since $X^t(i, w)$ is integer valued to show that:
\begin{align*}
    \mathbb{P}\left(T_{\max}(t) \geq \nu t\right)& = \mathbb{P}\left(\max_{w, \sum_{i=1}^n X_i^t(w) > 0 } w\geq \nu t\right) \leq \mathbb{P}\left(\cup_{w \geq \nu t } \sum_{i=1}^n X_i^t(w) \geq 1 \right)\leq \sum_{w \geq \nu t} \mathbb{P}\left(\sum_{i=1}^n X_i^t(w) \geq 1\right),
\end{align*}
so using Markov inequality yields: 
\begin{equation}
    \mathbb{P}\left(T_{\max}(t) \geq \nu t\right) \leq \sum_{w \geq \nu t} \sum_{i=1}^n \mathbb{E}\left[X_i^t(w)\right].
\end{equation}

\end{proof}

Variables $X_i^t$ are obtained by linear recursions, so Lemma~\ref{lemma:ptgttheta} allows us to bound the growth of variables with a simple recursion formula instead of evaluating a maximum. We write $p_i^{\rm comp}$ and $p_i^{\rm comm}$ the probability that node $i$ performs a computation (respectively communication) update at a given time step, and $p_i = p_i^{\rm comp} + p_i^{\rm comm}$. We introduce $\underbar{p}_{\rm comp} = \min_i p_i^{\rm comp}$ and $\bar{p}_{\rm comp} = \max_i p_i^{\rm comp}$ (and the same for communication probabilities).

\begin{lemma}
\label{lemma:sum_binom}
For all $i$, and all $\nu > 0$, if $\frac{1}{2} \geq \underbar{p}_{\rm comp} = \bar{p}_{\rm comp} \geq \bar{p}_{\rm comm}$ then:

\begin{equation}
    \sum_{w \geq \left(\nu_c + \nu_l\right) t} \sum_{i=1}^n \mathbb{E}\left[X^t(i, w)\right] \rightarrow 0 \text{ when } t \rightarrow \infty,
\end{equation}

with $\nu_c = 6 p_c \tau_c$ and $\nu_l = 9 p_l \tau_l$ where $p_c = 4 \bar{p}_{\rm comm}$ and $p_l = \bar{p}_{\rm comp}$.
\end{lemma}

Note that the constants in front of the $\nu$ parameters are very loose.

\begin{proof}

Taking the expectation over the edges that can be activated gives, with $\tau_c^{ij}(\tau)$ the probability that $\tau_c^{ij}$ takes value $\tau$ (and the same for $\tau_l$):

\begin{align*}
    \mathbb{E}\left[X^{t+1}(i,w)\right] = \left(1 - p_i\right) \mathbb{E}\left[X^{t}(i,w)\right] + &p_{\rm comm}\sum_{j=1}^n p_{ij} \sum_{\tau = 0}^\infty \tau_c^{ij}(\tau) \left(\mathbb{E}\left[X^{t}(i, w - \tau)\right] + \mathbb{E}\left[X^{t}(j, w - \tau)\right]\right) \\
    & + p_i^{\rm comp} \sum_{\tau = 0}^\infty \tau_l^{ij}(\tau) \mathbb{E}\left[X^{t}(i, w - \tau)\right].
\end{align*}

In particular, for all $i$, $\mathbb{E}\left[X^{t+1}(i,w)\right] \leq \bar{X}^t(w)$ where $\bar{X}^0(w) = 1$ if $w = 0$ and:

\begin{equation}
    \bar{X}^{t+1}(w) = \left(1 - \underbar{p} \right)\bar{X}^{t}(w) + 2 \bar{p}_{\rm comm} \sum_{\tau = 0}^\infty \tau_c^{\max}(\tau) \bar{X}^{t}(w - \tau) + \bar{p}_{\rm comp} \sum_{\tau = 0}^\infty \tau_l^{\max} (\tau) \bar{X}^{t}(w - \tau).
\end{equation}

with $\tau_c^{\max}(\tau) = \max_{ij} \tau_c^{ij}(\tau)$ (and the same for $\tau_l$). We now introduce $\phi^t(z) = \sum_{w \in \mathbb{N}} z^w \bar{X}^t(w)$. We denote $\phi_c$ and $\phi_l$ the generating functions of $\tau_c^{\max}(\tau)$ and $\tau_l^{\max}(\tau)$. A direct recursion leads to:
\begin{equation}
    \phi^t(z) = \left(1 - \underbar{p}_{\rm comm} - \underbar{p}_{\rm comp} +  \bar{p}_{\rm comp} \phi_l(z) + 2 \bar{p}_{\rm comm} \phi_c(z) \right)^t = \left(\phi^1(z)\right)^t.
\end{equation}

We denote $\phi_{bin}(p,t)$ the generating function associated with the binomial law of parameters $p$ and $t$. With this definition, we have:

\begin{equation}
    \phi_{bin}(p_c,t)(\phi_c(z)) \phi_{bin}(p_l,t)(\phi_l(z)) = \left[ (1 - p_c)(1 - p_l) + (1 - p_c)p_l \phi_l(z) + (1 - p_l)p_c \phi_c(z) + p_c p_l \phi_c(z) \phi_l(z)\right]^t,
\end{equation}
so we can define:

\begin{equation}
     \phi^t_+(z) = (1 + \delta)^t \phi_{bin}(p_c,t)(\phi_c(z)) \phi_{bin}(p_l,t)(\phi_l(z)),
\end{equation}

where $p_c$, $p_l$ and $\delta$ are such that:

\begin{equation*}
    \frac{p_c}{1 - p_c} \geq 2 \frac{\bar{p}_{\rm comm}}{1 - \underbar{p}}, \ \    \frac{p_l}{1 - p_l} = \frac{\bar{p}_{\rm comp}}{1 - \underbar{p}} \text{ and } \delta \geq \frac{1 - \underbar{p}}{(1 - p_c)(1 - p_l)} - 1. 
\end{equation*}

Since $\bar{p}_{\rm comp} = \underbar{p}_{\rm comp}$ then $\underbar{p} \geq \bar{p}_{\rm comp}$. Therefore, these conditions are satisfied for $p_c$ and $p_l$ as given by Lemma~\ref{lemma:sum_binom} and $\delta = (1 - p_c)^{-1} - 1$. Then $(1 + \delta)(1 - p_c)(1 - p_l) \geq 1 -  \underbar{p}$, $(1 + \delta)(1 - p_c)p_l \geq \bar{p}_{\rm comp}$ and  $(1 + \delta)(1 - p_l)p_c \geq 2\bar{p}_{\rm comm}$. This means that if we write $\phi^1(z) = a_0 + a_c \phi_c(z) + a_l \phi_l(z)$ and $\phi^1_+(z) = b_0 + b_c \phi_c(z) + b_l \phi_l(z)$ then $b_0 \geq a_0$, $b_c \geq a_c$ and $b_l \geq a_l$. In particular, all the coefficients of $\phi^t$ are smaller than the coefficients of $\phi^t_+$ where both functions are integral series. Therefore, if we call $Z_t$ the random variables associated with the generating function $(1 + \delta)^{-t}\phi^t_+$ then for all $i, t, w$:

\begin{equation}
\label{eq:control_EXt}
\mathbb{E}\left[X^t(i, w)\right] \leq (1 + \delta)^t \mathbb{P}\left(Z_t = w\right),
\end{equation}

where $Z_t = Z_c^t + Z_l^t =  Bin(p_c, t)(Z_c) + Bin(p_l, Z_l)(\tau_l)$ where $Z_c$ and $Z_l$ are the random variables modeling the time of one communication or computation update. We can then use the bound $p(Z_t \geq (\nu_c + \nu_l) t) \leq p(Z_c^t \geq \nu_c t) + p(Z_l^t \geq \nu_l t)$. This way, we can bound the \emph{communication} and \emph{computation} costs independently. Then, we write a Chernoff bound, i.e. for any $\lambda > 0$:

\begin{align*}
    \mathbb{P}\left(Z_c^t \geq \nu t \right) \leq e^{-\lambda \nu t} \esp{e^{\lambda Z_c^t}} = e^{-\lambda \nu t} \esp{e^{\lambda Z_c}}^t =e^{-\lambda \nu t} \left[1 - p_c + p_c \sum_{\tau = 0}^\infty p_c(\tau)e^{\lambda \tau} \right]^t,
\end{align*}

where $S_c$ is the sum of $t$ i.i.d. random variables drawn from $\tau_c$. If $Z_c = \tau_c$ with probability $1$ (deterministic delays) then this reduces to:
\begin{align*}
    \mathbb{P}\left(Z_c^t \geq \nu_c t \right) \leq e^{-\lambda \nu_c t} \left[1 - p_c + p_c e^{\lambda \tau_c} \right].
\end{align*}

Finally, we take $\nu_c = k p_c \tau_c$, $\lambda = \frac{1}{\tau_c}\ln(k)$ and we use the basic inequality $\ln(1 + x) \geq \frac{x}{1 + x}$ to show that:

\begin{equation}
    - \ln\left[\mathbb{P}\left(Z_c^t \geq \nu_c t \right)\right] \geq t\left[\lambda \nu_c - p_c \left(e^{\lambda \tau_c} - 1\right) \right] \geq t (k(\ln(k) - 1) - 1)p_c.
\end{equation}

Using the same log inequality and the fact that $p_c \geq \frac{1}{2}$ yields:
\begin{equation}
    \ln\left(1 + \delta\right) = - \ln(1 - p_c) \leq \frac{p_c}{1 - p_c} \leq 2 p_c.
\end{equation}

Therefore, choosing $k = 6$ ensures that  $k(\ln(k) - 1) - 1 \geq 3$ and so: 
\begin{equation}
    (1 + \delta)^t \mathbb{P}\left(Z_c^t \geq \nu_c t \right) \leq e^{-t p_c}.
\end{equation}

We can apply the same reasoning to $Z_l^t$, and the bound is still valid with $k = 9$ because $p_l = \bar{p}_{\rm comp} \geq \bar{p}_{\rm comm} = p_c / 4$. We finish the proof by using Equation~\eqref{eq:control_EXt}.
\end{proof}

\section{Algorithm Performances}
\label{app:algo_perfs}
\adfs~has a linear convergence rate because it results from using generalized APCG. Yet, it is not straightforward to derive hyperparameters that lead to a rate that is fast and that can be easily interpreted. The goal of this section is to choose such parameters when the functions $f_{i,j}$ are smooth. 

\subsection{Smallest eigenvalue of the Laplacian of the augmented graph}
The strong convexity of $q_A$ on the orthogonal of the kernel of $A$ is equal to $\sigma_A = \lambda_{\min}^+\left(A^T\Sigma^{-1}A\right)$, the smallest non-zero eigenvalue of $A^T\Sigma^{-1}A$. Indeed, $\Sigma^{-1}$ is a diagonal matrix with strictly positive entries only so ${\rm Ker}(A^T\Sigma^{-1}A) = {\rm Ker}(A)$. The goal of this section is to prove that for a meaningful choice of $\mu$, the smallest eigenvalue of the Laplacian of the augmented graph is not too small compared to the Laplacian of the actual graph. More specifically, we prove the following result:

\begin{lemma}
If for all virtual edge between a node $i$ and its virtual node $j$, $\mu_{ij}$ is such that $\mu_{ij}^2 = \frac{\lambda_{\min}^+(L)}{\sigma \kappa_i} L_{i,j}$ and $\sigma$, $\kappa$ are such that for all $i$, $\sigma \geq \sigma_i$ and $\kappa \geq \kappa_i$ then:

\begin{equation}
    \lambda_{\min}^+(\tilde{L}) \geq \frac{\lambda_{\min}^+(L)}{2 \sigma \kappa}.
\end{equation}
\end{lemma}

\begin{proof}
All non-zero singular values of a matrix $M^TM$ are also singular values of the matrix $MM^T$, and so $\lambda_{\min}^+(A^T\Sigma^{-1} A) = \lambda_{\min}^+ \left(\Sigma^{-1/2} A A^T \Sigma^{-1/2}\right)$. We denote $\tilde{L} = \Sigma^{-1/2} A A^T \Sigma^{-1/2}$.

Then, we denote $\mu_{ij}^2$ the weight of the \emph{virtual edge} $(i,j)$ and $M$ the diagonal matrix of size $nm$ which is such that ${e^{(i,j)}}^T M e^{(i,j)} = \mu_{ij}^2$ for all virtual nodes. $M_{n,m}$ is the matrix of size $n \times nm$ such that $(M_{n,m} e_{ij})^{(i)} = \mu_{ij}^2$ for all virtual edges $(i,j)$ and all other entries are equal to 0. Finally, $\tilde{S}$ is the diagonal matrix of size $n$ such that $\tilde{S}_i = \sum_{j =1}^m \mu_{ij}^2$. All \emph{communication nodes} are linked by the true graph, whereas all \emph{virtual nodes} are linked to their corresponding communication node. Then, if we denote $L$ the Laplacian matrix of the original true graph, the rescaled Laplacian matrix of the augmented graph writes:

\begin{equation}
\tilde{L} = \Sigma^{-1/2}\begin{pmatrix} L + \tilde{S} & - M_{n,m} \\ - M_{n,m}^T & M \end{pmatrix}\Sigma^{-1/2}.
\end{equation}

Therefore, if we split $\Sigma$ into two diagonal blocks $D_1$ (for the communication nodes) and $D_2$ (for the computation nodes) and apply the block determinant formula, we obtain: 

\begin{align*}
    &\det(D_1^{-\frac{1}{2}}A^TA D_1^{-\frac{1}{2}} - \lambda Id) = \det(D_2^{-1}M - \lambda Id)\\
    &\times \det(D_1^{-1/2} L D_1^{-1/2} + D_1^{-1} \tilde{S} - \lambda Id -\\
    &D_1^{-\frac{1}{2}} M_{n,m} D_2^{-\frac{1}{2}} \left(D_2^{-1} M - \lambda Id\right)^{-1} D_2^{-\frac{1}{2}} M_{n,m}^T D_1^{-\frac{1}{2}}).
\end{align*}

Then, we choose $M$ such that $D_2^{-1} M = diag(\alpha_1, ..., \alpha_n)$, meaning that $\mu_{ij}^2 = \alpha_i L_{i,j}$. With this choice, $D_1^{-\frac{1}{2}} M_{n,m} D_2^{-\frac{1}{2}} \left(D_2^{-1} M - \lambda Id\right)^{-1} D_2^{-\frac{1}{2}} M_{n,m}^T D_1^{-\frac{1}{2}}$ is a diagonal matrix where the $i$-th coefficient is equal to

\begin{equation}
    \frac{1}{\sigma_i}\sum_{j=1}^m \mu_{ij}^4 \frac{1}{\mu_{ij}^2 - L_{i,j} \lambda} = \kappa_i \frac{\alpha^2_i}{\alpha_i - \lambda},
\end{equation}

where $S_i = \sum_{j =1}^m L_{i,j}$ and $\kappa_i = 1 + \frac{S_i}{\sigma_i}$. On the other hand, $D_1^{-1}S$ is also a diagonal matrix where the $i$-th entry is equal to $\alpha \kappa_i$. Therefore, the solutions of $det(\tilde{L} - \lambda Id) = 0$ are $\lambda = \alpha_i$ and the solutions of:
\begin{equation}
    \label{eq:det_deltalambda}
    det(D_1^{-1/2} L D_1^{-1/2} - \Delta_\lambda) = 0
\end{equation}

with $\Delta_\lambda$ a diagonal matrix such that $(\Delta_\lambda)_{i,i} = \left(\frac{1}{\sigma_i}\sum_{j=1}^{m} \mu_{ij}^4\left( \frac{1}{\mu_{ij}^2 - L_{i,j} \lambda} - 1  \right) + \lambda \right)$. All the entries of $\Delta_\lambda$ grow with $\lambda$, meaning that the smallest solution $\lambda^*$ of Equation~\eqref{eq:det_deltalambda} is lower bounded by the smallest solution of:

\begin{equation}
    \det\left(\frac{\lambda_{\min}^+(L)}{\sigma}Id - \Delta_\lambda\right).
\end{equation}

If $\alpha_i > 0$ and we choose $\lambda \neq \alpha_i$, then the other singular values of $\tilde{L}$ are lower bounded by the minimum over all $i$ of the solution of:

\begin{equation}
    \nu - \left(\frac{1}{\sigma_i}\sum_{j=1}^{m} \mu_{ij}^4\left( \frac{1}{\mu_{ij}^2 - L_{i,j} \lambda} - 1  \right) + \lambda \right) = 0,
\end{equation}

where $\nu = \frac{\lambda_{\min}^+(L)}{\sigma}$ which, with our choice of $\mu_{ij}$ gives:

\begin{equation}
    \nu - \left(\alpha_i \frac{S_i}{\sigma_i}\left(\frac{\alpha_i}{\alpha_i - \lambda } - 1\right) + \lambda \right) = 0,
\end{equation}

that can be rewritten:
\begin{equation}
    \lambda^2 - \lambda\left(\nu + \alpha_i \kappa_i\right) + \alpha_i \nu  = 0.
\end{equation}

Therefore, noting $\lambda_i^*$ the smallest solution of this system for a given $i$, we get:

\begin{equation}
    \lambda_i^* \geq \frac{1}{2}\left(\alpha_i\kappa_i + \nu - \sqrt{\left(\nu + \alpha_i\kappa_i \right)^2 - 4\nu \alpha_i}\right).
\end{equation}

In particular, we choose $\alpha_i = \frac{\nu}{\kappa_i}$ and use that $\sqrt{1 - x} \leq 1 - \frac{x}{2}$ to show:
\begin{equation}
    \lambda_i^*  \geq \nu \left(1 - \sqrt{1 - \kappa_i^{-1}}\right)
    \geq \frac{\nu}{2\kappa_i}.
\end{equation}

The other eigenvalues are given by the values that zero out diagonal terms of the lower right corner. These are the solutions of $\mu_{ij}^2 = L_{i,j} \lambda$, yielding $\lambda = \alpha_i \geq \lambda_i^*$. Therefore, $\lambda_{\min}^+(\tilde{L}) \geq \min_i \lambda_i^*$, which finishes the proof.

\end{proof} 

\subsection{Eigengap of the augmented graph}
\label{app:gamma_tilde}
This section aims as justifying the $\tilde{\gamma}$ notation. Recall that it is defined such that $\tilde{\gamma} = \min_{(k,\ell) \in E^{\rm comm}} \frac{\lambda_{\min}^+(L)n^2}{\mu_{k\ell}^2 e_{k\ell}^TA^\dagger Ae_{k\ell}E^2}$. We show in this section that for any given family of regular graphs, there exists a constant $C_\gamma$ independent of the size of the graph such that $C_\gamma \tilde{\gamma} \geq \gamma$. Matrix $A$ depends on $\mu$, and we consider in this section that $\mu_{k\ell}^2 = \mu_0$ for all communication edges $(k,\ell)$. Similar results can be obtained when $\mu$ is heterogeneous.

\paragraph{Regular graphs.} We say that a family of graph is regular if there exists $C_\gamma > 0$
such that $e_{k\ell}^TA^\dagger Ae_{k\ell} \leq C_\gamma \frac{n}{E}$ for any $n > 2$.

Recall that $E$ is the number of edges (usually constrained by the graph family and the number of nodes), and $e_{k\ell}^TA^\dagger Ae_{k\ell}$ is the effective resistance of edge $(k, \ell)$. 
This assumption seems a bit technical but it simply requires that all edges contribute equally to the connectivity of the graph, and therefore is related to how symmetric the graph is. In particular, it is verified with $C_\gamma = 1$ for any completely symmetric graph, such as the complete graph or the ring. Since $e_{k\ell}^T A^\dagger A e_{k\ell} \leq 1$, it is also satisfied any time the ratio $n / E$ is bounded below, and in particular for the grid, the hypercube, or any graph with bounded degree. Under these assumptions, and for any communication edge $(k,\ell)$: 
$$\frac{\lambda_{\min}^+(L)n^2}{\mu_{k\ell}^2 e_{k\ell}^TA^\dagger Ae_{k\ell}E^2} \geq \frac{\gamma}{C_\gamma} \frac{\lambda_{\max}(L)n}{\mu_{k\ell}^2 E} \geq \frac{\gamma}{C_\gamma} \frac{{\rm Trace}(L)}{\mu_{0}^2 E} = 2 \frac{\gamma}{C_\gamma}.$$
Here, we used the fact that ${\rm Trace}(L) = 2 \mu_0^2 E$, which can be deduced directly from the form of $A$ (each edge has weight $\mu_0^2$ and contributes two times, one for each end). We conclude by using the fact that since the previous inequalities are true for any $(k,\ell)\in E^{\rm comm}$, it is in particular true for $\tilde{\gamma}$.

\subsection{Communication rate and local rate}
We know that the rate of ADFS can be written as the minimum of a given quantity over all edges of the graph. This quantity will be very different whether we consider communication edges or virtual edges. In this section, we give lower bounds for each type of edge, and show that we can trade one for the other by adjusting the probability of communication. 

\begin{lemma}
With the choice of parameters of Theorem~\ref{thm:adfs_speed}, parameter $\rho$ satisfies:
\begin{equation}
    \rho \geq \frac{1}{\sqrt{2} n}\min\left( p_{\rm comm} \Delta_p \sqrt{\frac{\tilde{\gamma}}{2 \kappa}}, p_{\rm comp}  \frac{\sqrt{r_\kappa}}{S_{\rm comp}} \right).
\end{equation}
\end{lemma}

\begin{proof}
Recall that the rate $\rho$ is defined as:
\begin{equation}
\label{eq:rate_theta}
\rho^2 = \min_{k\ell} \frac{p_{k\ell}^2}{\mu_{k\ell}^2 e_{k\ell}^T A^\dagger A e_{k\ell}} \frac{\lambda_{\min}^+ ( \tilde{L} )}{\sigma_k^{-1} + \sigma_\ell^{-1}}.
\end{equation}

Therefore, for communication edges the rate writes:
\begin{equation}
\rho^2_{\rm comm} \geq \min_{(k,\ell) \in E^{\rm comm}} \left(\frac{1}{\sigma_k} + \frac{1}{\sigma_\ell}\right)^{-1} \frac{p_{k\ell}^2}{\mu_{k\ell}^2e_{k\ell}^TA^\dagger Ae_{k\ell}} \frac{\lambda_{\min}^+ \left( L\right)}{2\sigma  \kappa}.
\end{equation}
If we take $\sigma_k = \sigma$ for all $k$ and $p_{k\ell}^{2} \geq \Delta_p^2 p_{\rm comm}^2 / |E|^2$ (corresponding to a homogeneous case), then we can make $\tilde{\gamma}$ appear to obtain:
\begin{equation}
\label{eq:lb_comm}
\rho^2_{\rm comm} \geq \Delta_p^2 \frac{\tilde{\gamma}}{\kappa} \frac{p_{\rm comm}^2}{4n^2}.
\end{equation}

For ``computation edges", we can write:
\begin{equation}
\rho_{\rm comp}^2 \geq \min_{ij} \frac{p_{ij}^2}{2\left(\sigma_i^{-1} + L_{i,j}^{-1}\right)}\frac{\sigma \kappa_i}{\lambda_{\min}^+ \left(L\right)L_{i,j}} \frac{\lambda_{\min}^+ \left(L\right)}{\sigma \kappa},
\end{equation}

because $e_{ij}^T A^\dagger A e_{ij} = 1$ when $(i,j)$ is a virtual edge (because it is part of no cycle). Since $S_{\rm comp} = \frac{1}{n}\sum_{i=1}^n \sum_{j=1}^{m} \sqrt{1 + L_{i,j}\sigma_i^{-1}}$, this can be rewritten:
\begin{equation}
\rho^2_{\rm comp} \geq \frac{r_\kappa}{2} \frac{p_{\rm comp}^2}{n^2S_{\rm comp}^2}. 
\end{equation}
\end{proof} 

\subsection{Execution time}
Now that we have specified the rate of ADFS (improvement per iteration), we can bound the time needed to reach precision $\varepsilon$ by plugging in the expected time to execute the schedule. In particular, we show in this section Theorem~\ref{thm:adfs_speed_precise}, which is a more precise version of Theorem~\ref{thm:adfs_speed}.

We introduce $\Delta_p$, $r_\kappa$ and $c_\tau$ to quantify how heterogeneous the system is. More specifically, we can define $\sigma = \max_i \sigma_i$, $\kappa_i = 1 + \sigma_i^{-1}\sum_{j =1}^m L_{i,j}$ and $\kappa_s = \max_i \kappa_i$. Since they are not all equal, we introduce $r_\kappa =  \min_i \kappa_i / \kappa_s$. We choose the probabilities of virtual edges, such that $\sum_{j=1}^m p_{ij}$ is constant for all $i$ and such that $p_{ij} = p_{\rm comp} (1 + L_{i,j} \sigma_i^{-1})^{\frac{1}{2}} / (n S_{\rm comp})$ for $S_{\rm comp} = n^{-1}\sum_{i=1}^n\sum_{j=1}^m (1 + L_{i,j} \sigma_i^{-1})^{\frac{1}{2}}$. When $(k,\ell)$ is a communication edge, we further assume that $p_{k\ell} \geq \Delta_{\rm p}  p_{\rm comm} / |E|$ for some constant $\Delta_{\rm p} \leq 1$ and $p^{\max}_{\rm comm} \leq c_\tau p_{\rm comm}$ for some $c_\tau > 0$.
\begin{theorem}
\label{thm:adfs_speed_precise}
We choose $\mu_{k\ell}^2 = \frac{1}{2}$ for communication edges, $\mu_{ij}^2 = \frac{\lambda_{\min}^+(L)}{\sigma \kappa_i}L_{i,j}$ for computation edges and $p_{\rm comm} = \min\Big(\frac{1}{2}, \left(1 + \sqrt{\frac{\tilde{\gamma}}{\kappa_{min}}}S_{\rm comp}\right)^{-1}\Big)$. Then, running Algorithm~\ref{algo:sc_adfs} for $K = \rho^{-1}\log\left(\varepsilon^{-1}\right)$ iterations guarantees $\mathbb{E}\left[\| \theta_K - \theta^\star\|^2 \right] \leq C_0 \varepsilon$, and takes time $T(K)$, with $T(K)$ bounded by:
\begin{equation*}
    T(K) \leq 2C\left(\frac{m + \sqrt{m\kappa_s}}{\sqrt{2 r_\kappa}} + \frac{\left(1 + 4 c_\tau \tau\right)}{\Delta_{\rm p}}\sqrt{\frac{\kappa_s}{\tilde{\gamma}}} \right) \log\left(\frac{1}{\varepsilon}\right)
\end{equation*}
with probability tending to $1$ as $\rho^{-1}\log\left(\varepsilon^{-1}\right) \rightarrow \infty$, where $C$ is the same as in Theorem~\ref{thm:synchronization_cost}.
\end{theorem}
In heterogeneous settings, $\sigma_i$ and sampling probabilities may be adapted to recover good guarantees, but this is beyond the scope of this paper. Note that taking computing probabilities exactly equal for all nodes is not necessary to ensure convergence, and only slightly slows down convergence. Indeed, it is always possible to analyze a schedule for which all nodes have exactly the same probability of local update by adding a probability of doing nothing for time $1$ as a local update to the nodes that are chosen less frequently. If we denote $p_{\rm wait}$ the probability that we need to add so that all nodes have the same probability of being selected, then $p_{\rm comp} + p_{\rm comm} = 1 - p_{\rm wait}$ so $\theta_{\rm comp}$ will be slightly smaller for a given $p_{\rm comm}$. The actual algorithm can only be faster so this just gives a rough upper bound on the time to convergence.

\begin{proof}
Using Theorem~\ref{thm:synchronization_cost} on the average time per iteration, we know that as long as $p_{\rm comp} > p_{\rm comm}$, the execution time of the algorithm verifies the following bound for some $C > 0$ with high probability:

\begin{equation}
    T(K) \leq \frac{C}{n} \left(p_{\rm comp} +  2\tau p^{\max}_{\rm comm}\right) K
\end{equation}

Algorithm~\ref{algo:sc_adfs} requires $ - \log(1 / \varepsilon) / \log(1 - \rho)$ iterations to reach error $\varepsilon$. Using that $\log(1 + x) \leq x$ for any $x > -1$, we get that using $K_\varepsilon = \log(1 / \varepsilon) \rho^{-1}$ instead also guarantees to make error less than $\varepsilon$. We now optimize the bound in $\rho$:

\begin{equation}
    \frac{T(K_\varepsilon)}{\log\left(\varepsilon^{-1}\right)} \leq \frac{C}{n\rho} \left(p_{\rm comp} + 2\tau p^{\max}_{\rm comm}\right)
\end{equation}

If we rewrite this in terms of $\rho_{\rm comm}$ and $\rho_{\rm comp}$, we obtain:

\begin{equation}
     \frac{T(K_\varepsilon)}{\log\left(\varepsilon^{-1}\right)} \leq C \max \left(T_1(p_{\rm comm}), T_2(p_{\rm comm})\right)
\end{equation}

with 
\begin{equation}
    T_1(p_{\rm comm}) = \frac{1}{n \rho_{\rm comm}}(p_{\rm comp} + 2c_\tau\tau p_{\rm comm}) = \frac{2}{\Delta_p}\left(2c_\tau \tau - 1 + \frac{1}{p_{\rm comm}}\right)\sqrt{\frac{\kappa}{\tilde{\gamma}}}
\end{equation}

and 
\begin{equation}
    T_2(p_{\rm comm}) = S_{\rm comp}\sqrt{\frac{2}{r_\kappa}}  \left( \frac{1 + (2c_\tau \tau - 1)p_{\rm comm}}{1 - p_{\rm comm}}\right) = S_{\rm comp}\sqrt{\frac{2}{r_\kappa}} \left(1 + 2\tau \frac{p_{\rm comm}}{1 - p_{\rm comm}}\right)
\end{equation}

$T_1$ is a continuous decreasing function of $p_{\rm comm}$ with $T_1 \rightarrow \infty$ when $p_{\rm comm} \rightarrow 0$. Similarly, $T_2$ is a continuous increasing function of $p_{\rm comm}$ such that $p_{\rm comm} \rightarrow \infty$ when $p_{\rm comm} \rightarrow 1$. Therefore, the best upper bound on the execution time is given by taking $p_{\rm comm} = p^*$ where $p^*$ is such that $T_1(p^*) = T_2(p^*)$ and so $\rho_{\rm comm}(p^*) = \rho_{\rm comp}(p^*)$.

\begin{equation}
    \frac{T(K_\varepsilon)}{\log\left(\varepsilon^{-1}\right)} \leq C T_1(p^*)
\end{equation}

Then, $p^*$ can be found by finding the root in $]0, 1[$ of a second degree polynomial. In particular, $p^*$ is the solution of:

\begin{equation}
    p_{\rm comp}^2 = p_{\rm comm}^2 \frac{\tilde{\gamma} \Delta_p^2}{2 \kappa r_\kappa}S_{\rm comp}^2 = (1 - p_{\rm comm})^2
\end{equation}

which leads to $p^* = \left(1 + \sqrt{\frac{\tilde{\gamma}}{2\kappa_{\min}}}\Delta_p S_{\rm comp}\right)^{-1}$.

\begin{align*}
    \frac{T(K_\varepsilon)}{\log\left(\varepsilon^{-1}\right)} &\leq 2 \frac{C}{\Delta_p} \left(2c_\tau \tau - 1 + \frac{1}{p^*}\right) \sqrt{\frac{\kappa}{\tilde{\gamma}}} \\
    &\leq 2C \left(2 \tau \frac{c_\tau}{\Delta_p}\sqrt{\frac{\kappa}{\tilde{\gamma}}} + \frac{1}{\sqrt{2 r_\kappa}} S_{\rm comp}
    \right)
\end{align*}

Finally, we use the concavity of the square root to show that:

\begin{align*}
    S_{\rm comp} &= \frac{1}{n}\sum_{i=1}^n \sum_{j=1}^{m} \sqrt{1 + L_{i,j}\sigma_i^{-1}}\\
    &\leq \frac{1}{n}\sum_{i=1}^n m \sqrt{\sum_{j=1}^{m} \frac{1}{m}\left(1 + L_{i,j}\sigma_i^{-1}\right)}\\
    &\leq \frac{1}{n}\sum_{i=1}^n m \sqrt{1 + \frac{1}{m}(\kappa_i - 1)}\\
    &\leq m + \sqrt{m\kappa}
\end{align*}

Yet, this analysis only works as long as $p^* \leq 1/2$. When this constraint is not respected, we know that: $\tilde{\gamma}S_{\rm comp}^2 \leq 2 \kappa r_\kappa$. In this case, we can simply choose $p_{\rm comp} = p_{\rm comm} = \frac{1}{2}$ and then $\rho_{\rm comm} \leq \rho_{\rm comp}$, so 

\begin{equation}
    \frac{T(K_\varepsilon)}{\log\left(\varepsilon^{-1}\right)} \leq C T_1\left(\frac{1}{2}\right) = 2 \frac{C}{\Delta_p} \left(1 + 2 c_\tau \tau\right)\sqrt{\frac{\kappa}{\tilde{\gamma}}}
\end{equation}

The sum of the two bounds is a valid upper bound in all situations, which finishes the proof.
\end{proof} 

\section{Experimental setting}
\label{app:experimental_setting}
\subsection{Experimental Setting}
We detail in this section the exact experimental setting in which simulations were made. All algorithms used out-of-the-box parameters given by theory. Batch algorithms as well as ESDACD were given the exact $\kappa_b$. The datasets we used are the first million samples of the Higgs dataset (11 million samples and 28 attributes) and the Covtype.binary.scale dataset (581,012 samples and 54 attributes). Both datasets are available at \url{https://www.csie.ntu.edu.tw/~cjlin/libsvmtools/datasets/binary.html}. To obtain the local dataset $X_i \in \mathbb{R}^{m \times d}$ of each node, we drew $m$ samples at random from the base dataset, so that datasets of different nodes may overlap. We used the logistic loss with quadratic regularization, meaning that the function at node~$i$ is:
$$f_i: \theta_i \mapsto \sum_{j=1}^m \log\left(1 + \exp(-l_{i,j} X_{i,j}^T\theta_i)\right) + \frac{\sigma_i}{2}\|\theta_i \|^2,$$
where $l_{i,j} \in \{-1, 1\}$ is the label associated with $X_{i,j}$, the $k$-th sample of node $i$. We chose $m = 10^4$ and $\sigma = 1$ for all simulations. Note that local functions are not normalized (not divided by $m$) so this actually corresponds to a regularization value of $\sigma_i = 10^{-4}$ with usual formulations. Computation delays were chosen constant equal to $1$ and communication delays constant equal to $5$.

As said in the main text, plots are shown for \emph{idealized times} in order to abstract implementation details as well as ensure that reported timings were not impacted by the cluster status (available bandwidth for example).
Note that for ADFS, nodes perform the schedule described in Section~\ref{sec:synch_time} and are considered free to start the next iteration as soon as they send their a gradient as long as they already received the neighbor's gradient (non-blocking send). Note that although Algorithm~\ref{algo:sc_adfs} returns vector $\Sigma^{-1}v_t$ to compute the error, we used the vector $\Sigma^{-1}y_t$ instead. Both have similar asymptotic convergence rates but the error was more stable using $\Sigma^{-1}y_t$. The error that we plot is the average error over all nodes at a given time. More specifically, all nodes compute the error at specific iteration number as $F(\Sigma^{-1}y_t)$. Then, we average all these errors and the time reported is the time at which the last node finishes this iteration.

Similarly to Table~\ref{fig:table_speeds}, we assume that computing the dual gradient of a function $f_i$ is as long as computing $m$ proximal operators of $f_{i,j}$ functions. This greatly benefits to MSDA and ESDACD since in the case of logistic regression, the proximal operator for one sample has no analytic solution but can be efficiently computed as the result of a one-dimensional optimization problem~\citep{shalev2013stochastic}. The inner problem corresponding to computing $\nabla f_i^*$ was solved by performing $500$ steps of accelerated gradient descent. For Point-SAGA, ADFS and DSBA, 1D prox were computed using 5 steps of Newton's method (in one dimension). Both used warm-starts, \emph{i.e.} the initial parameter for these inner problems was the solution for the last time the problem was solved. The step-size $\alpha$ of DSBA was chosen as $1 / (4L_{\max})$ instead of $1 / (24L_{\max})$ where $L_{\max} = \max_{i,j} L_{i,j}$. DSBA was unstable for larger values of $\alpha$.

\subsection{Centralized Algorithms}
In this section, we perform a quick comparison with the centralized algorithm Katyusha~\citep{allen2017katyusha}. We assume that the allreduce communication steps take time $\Delta$ where $\Delta$ is the diameter of the graph. We implement the mini-batch version of this algorithm and set the mini-batch size so that $b = 1 + \Delta \tau$, \emph{i.e.}, the algorithm spends as much time computing as communicating (not counting the full gradient steps). Counting computation time in terms of effective passes over the dataset is slightly unfair to Katyusha that has a cheaper per-example cost. Yet, this is only a (small) constant factor in the case of logistic regression.

\begin{figure*}
\centering
\begin{subfigure}{.33\textwidth}
  \centering
  \includegraphics[width=\linewidth]{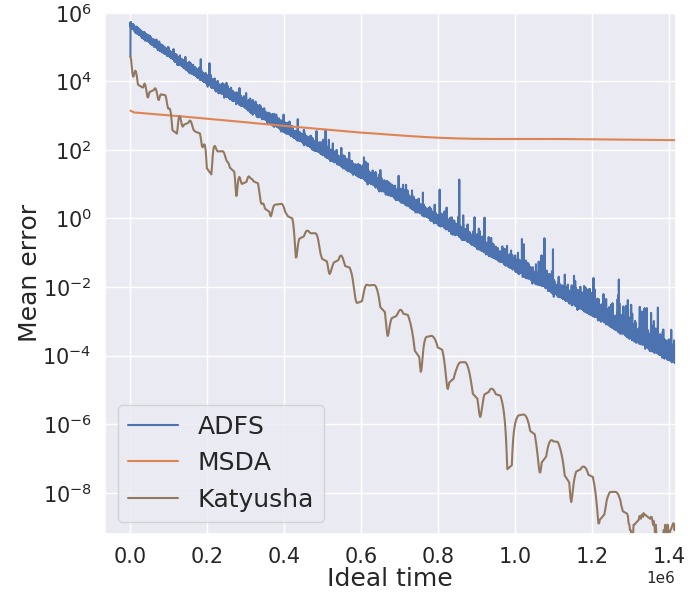}
  \vspace{-15pt}
\caption{Higgs, $\tau=5$}
\label{fig:adfs_kat_tau_5_higgs}
\end{subfigure}%
\begin{subfigure}{.33\textwidth}
  \centering
  \includegraphics[width=\linewidth]{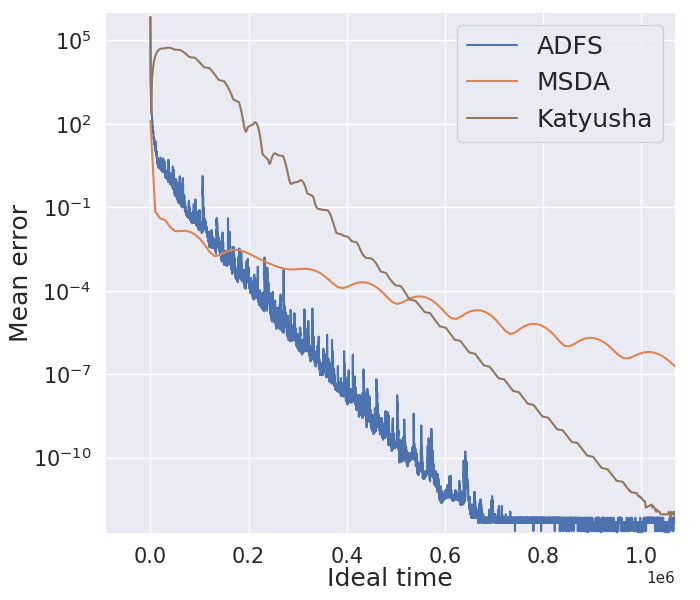}
  \vspace{-15pt}
\caption{Covtype, $\tau=5$}
\label{fig:adfs_kat_tau_5_covtype}
\end{subfigure}%
\begin{subfigure}{.33\textwidth}
  \centering
  \includegraphics[width=\linewidth]{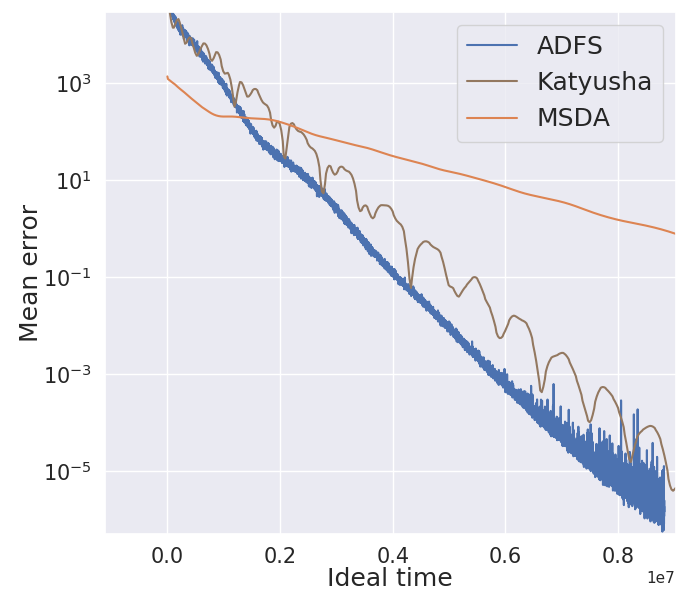}
  \vspace{-15pt}
  \caption{Higgs, $\tau=50$}
\label{fig:adfs_kat_tau_50}
\end{subfigure}
\vspace{5pt}
\caption{Simulations on the logistic regression task with $m = 10^4$ points per node, regularization parameter $\sigma=1$ on grid networks of size $100$.}
\label{fig:katyusha_adfs_simulations}
\end{figure*}
We observe on Figures~\ref{fig:adfs_kat_tau_5_higgs} and \ref{fig:adfs_kat_tau_5_covtype} that Katyusha and ADFS have comparable rates when $\tau = 5$. This is mainly due to the fact that communications are quite fast so the effective mini-batch size is $9100$ in this case (diameter of the graph is $18$ so $91$ per node), which is quite small compared to the $10^6$ total samples. Figure~\ref{fig:adfs_kat_tau_50} shows that ADFS can outperform Katyusha on the Higgs dataset (on which it was slower when taking $\tau=5$) when delays are big ($\tau = 50)$. Indeed, the effective batch size in this case is $91000$, which is about $10\%$ of the dataset and so Katyusha does not take full advantage of the stochastic optimization speedup. Yet, it is still significantly faster than MSDA. Note that in the case of $\tau = 50$, we set $p_{\rm comm}$ such that $\tau p_{\rm comm} = p_{\rm comp}$. This choice slightly reduced the number communications and led to a faster algorithm by reducing communication time and synchronization barriers. Overall, we see that, contrary to existing decentralized methods, ADFS can be competitive with a distributed implementation of Katyusha, especially when delays are high. Yet, a more detailed study reporting actual computing times with fully optimized implementations would be needed to compare the algorithms further. Indeed, some simulation choices favored ADFS (normalized time, neglecting overhead induced by the prox), whereas other favored Katyusha (constant delays, homogeneous setting).

More fundamentally, Katyusha and ADFS are based on two distinct distribution paradigms. On the one hand, centralized algorithms use less noisy gradients because they have an effective mini-batch size of at least $n$. This grants them linear speedup given that the batch size is small enough. Yet, the batch size usually has to be quite high because it needs to grow linearly with the communication time and the diameter of the graph in order to avoid spending more time communicating than computing so centralized approaches are not necessarily the best option on high-latency networks. On the other hand, decentralized algorithms such as ADFS can work with very small batches but they do not get the mini-batch noise reduction from computing on $n$ nodes in parallel the way Katyusha does. Indeed, similarly to ``Local-SGD"~\cite{lin2018don,patel2019communication} approaches, each node locally runs an accelerated variance-reduced algorithm. This confirms that decentralized algorithms, and in particular ADFS, can be well-suited for distributed stochastic optimization with delays.

\subsection{Code}
A Python implementation of ADFS is given in supplementary material. The goal of this code is to show how to implement ADFS and encourage its use as a baseline. The code implements ADFS to solve the Logistic Regression problem on a 2D grid. It generates a synthetic binary classification dataset. Our implementation leverages the fact that Logistic Regression is a linear model to only store 2 scalars per virtual node instead of 2 full vectors, thus showing how to use sparse updates. The code is not optimized and not intended to be particularly fast, but rather to show how to go from the pseudo-code in Algorithm~\ref{algo:sc_adfs} to an actual implementation.

\end{document}